\documentclass[11pt]{article}
\usepackage{latexsym,amsfonts,amssymb,amsmath,amsthm}
\usepackage{graphicx}

\usepackage[usenames,dvipsnames]{color}
\usepackage{ulem,comment}

\usepackage{color}
\usepackage{array}
\begin{document}

\newtheorem{tm}{Theorem}[section]
\newtheorem{prop}[tm]{Proposition}
\newtheorem{defin}[tm]{Definition} % definition numbers are dependent on theorem numbers
\newtheorem{coro}[tm]{Corollary}
\newtheorem{lem}[tm]{Lemma}
\newtheorem{assumption}[tm]{Assumption}
\newtheorem{rk}[tm]{Remark}
\newtheorem{nota}[tm]{Notation}
\numberwithin{equation}{section}

\newcommand{\stk}[2]{\stackrel{#1}{#2}}
\newcommand{\dwn}[1]{{\scriptstyle #1}\downarrow}
\newcommand{\upa}[1]{{\scriptstyle #1}\uparrow}
\newcommand{\nea}[1]{{\scriptstyle #1}\nearrow}
\newcommand{\sea}[1]{\searrow {\scriptstyle #1}}
\newcommand{\csti}[3]{(#1+1) (#2)^{1/ (#1+1)} (#1)^{- #1
 / (#1+1)} (#3)^{ #1 / (#1 +1)}}
\newcommand{\RR}[1]{\mathbb{#1}}

\newcommand{\rd}{{\mathbb R^d}}
\newcommand{\ep}{\varepsilon}
\newcommand{\rr}{{\mathbb R}}
\newcommand{\alert}[1]{\fbox{#1}}
\newcommand{\eqd}{\sim}
\def\p{\partial}
\def\R{{\mathbb R}}
\def\N{{\mathbb N}}
\def\Q{{\mathbb Q}}
\def\C{{\mathbb C}}
\def\l{{\langle}}
\def\r{\rangle}
\def\t{\tau}
\def\k{\kappa}
\def\a{\alpha}
\def\la{\lambda}
\def\De{\Delta}
\def\de{\delta}
\def\ga{\gamma}
\def\Ga{\Gamma}
\def\ep{\varepsilon}
\def\eps{\varepsilon}
\def\si{\sigma}
\def\Re {{\rm Re}\,}
\def\Im {{\rm Im}\,}
\def\E{{\mathbb E}}
\def\P{{\mathbb P}}
\def\Z{{\mathbb Z}}
\def\D{{\mathbb D}}
\newcommand{\ceil}[1]{\lceil{#1}\rceil}

%\allowdisplaybreaks
\title{ Existence of traveling wave solutions of a deterministic vector-host epidemic model with direct transmission }

\author{
Dawit Denu\footnote{Department of Mathematics, Georgia Southern University, Savannah, GA 31419 (ddenu@georgiasouthern.edu)},
  Sedar Ngoma\footnote{Department of Mathematics, State University of New York at Geneseo, NY 14454 (ngoma@geneseo.edu)},
and Rachidi B. Salako\footnote{Department of Mathematics,
Ohio State University, Columbus, OH 43210 (salako.7@osu.edu)} }

\date{}
\maketitle
%\begin{document}

\begin{abstract} We consider an epidemic model with direct transmission given by a system of nonlinear partial differential equations and study the existence of traveling wave solutions. When the basic reproductive number of the considered model is less than one, we show that there is no nontrivial traveling wave solution. On the other hand, when the basic reproductive number is greater than one, we prove that there is a minimum wave speed $c^*$ such that the system has a traveling wave solution with speed $c$ connecting both equilibrium points for any $c\ge c^*$. Moreover, under suitable assumption on the diffusion rates, we show that there is no traveling wave solution with speed less than $c^*$. We conclude with numerical simulations to illustrate our findings. The numerical experiments supports the validity of our theoretical results. 

\end{abstract}



\section{Introduction}

According to several reports by different health organizations (e.g. Center for Disease Control and Prevention (CDC), National Notifiable Diseases Surveillance System (NNDSS), etc.. ), one of the leading causes for the death of children, adolescents and adults is infectious diseases. Among the different infectious diseases, a large proportion is transmitted by vectors such as mosquitoes, ticks, sand flies and others. In general, vector-borne diseases are infections transmitted by the bite of blood-feeding arthropods, collectively called vectors, or through contaminated urine, tissues or bites of infected animals such as rats or dogs. It is also known that some of the vector-borne infections can be transmitted directly whenever there is a physical contact with blood or body fluids between an infected person and a susceptible one \cite{lanata2013_global, lemon2008vector, tabachnick2010challenges}.

Vector-borne diseases have continued to be one of the most challenging  threats to human health, partly because transmission of the infection is directly related to a broad and complex external environmental factors such as climate change, changing ecosystems and landscape, population migration and other factors.

In order to understand how fast an infectious disease can be spread, how long the disease can exist and thus come up with the best strategies to stop the
spread of the disease, chose a better effective immunization program, allocate scarce resources to control or prevent infections and also predict the future course of an outbreak, mathematical modeling of epidemics is very vital. To that end, many authors have proposed and studied different mathematical  models of vector-host epidemics \cite{ArVan2006, ArVan2003a, ArVan2003b, Barlow2000,  Brauer2001, Brauer2008, Castillo2001, Hsieh2007, SaVan2006,  Wang_Mu2003, WaZh2006}.

Barlow \cite{Barlow2000} presented a mathematical model for a possum-tuberculosis (TB) system that is both realistic and parsimonious. In \cite{ArVan2006, ArVan2003a, ArVan2003b}, Arino et al.  proposed epidemic models with populations traveling among cities in which the residences of individuals are maintained.  Wang  et al. \cite{WaZh2006} formulated an epidemic model with population dispersal and infection period. Salmani et al.  \cite{SaVan2006} discussed an SEIRS epidemic model on patches to describe the dynamics of an infectious disease in a population in which individuals travel between patches. 

Brauer et al. \cite{Brauer2001} constructed and analyzed some simple models for disease transmission that include immigration of infective individuals and variable population size. 
Moreover, Castillo-Chavez et al.  \cite{Castillo2001} illustrate  the richness generated by discrete-time susceptible-infective-susceptible (S-I-S) disease transmission models in the study of two patch epidemic models with disease-enhanced or disease-suppressed dispersal.
Furthermore, Hsieh et al. \cite{Hsieh2007} proposed a multi-patch model to study the impact of travel on the spatial spread of disease between patches with different level of disease prevalence.
Finally, Wang et al. \cite{Wang_Mu2003} proposed an epidemic model to describe the dynamics of disease spread between two patches due to population dispersal. 

In particular, Cai and Li \cite{CaiLi} analyzed a generalized vector-host epidemic model with direct transmission. This model can be applied to most of the infections caused by vectors such as malaria, Zika virus infection, dengue fever and West Nile virus. In order to derive the model, let $x_1(t)$ and $x_2(t)$  represent the number of susceptible and infected hosts, respectively, and $x_3(t)$ and $x_4(t)$  represent the number of susceptible and infected vectors at any time $t\geq 0$, respectively. Assume that susceptible hosts can be infected both directly through contact with an infected host, such as blood transfusion, and indirectly by a bite from an infected vector, such as a mosquito. Similarly, we assume that if a susceptible vector bites an infected host, it will acquire the disease. The model does not assume disease-induced deaths in both species, that is, no one has died from the disease in the given time. The picture below depicts the transmission cycle of the vector-host epidemic model.

\begin{figure}[ht]  \centering  \includegraphics[width=3.2 in]{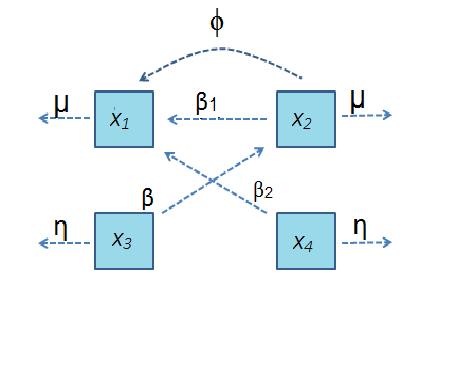} \caption{Vector host epidemic model with direct transmission.}\end{figure}

Based on the law of mass action, the dynamics of the vector-host epidemics is described by the following 4 coupled nonlinear ordinary differential equations.

\begin{equation}\label{det eqn 1}
\begin{cases}
\frac{dx_1}{dt} =b_1 -\mu x_1 -\beta _2  x_1 x_4-\beta_1 x_1 x_2 + \phi x_2, \cr
\frac{dx_2}{dt}=\beta_1 x_1 x_2 +\beta_2  x_1 x_4-(\mu + \phi)x_2, \cr
\frac{dx_3}{dt}=b_2 -\eta x_3-\beta x_3 x_2, \cr
\frac{dx_4}{dt}=\beta x_3 x_2-\eta x_4,
\end{cases}
\end{equation}
given the initial conditions
$$x_1(0) = x_1 ^0 >0, \, \,  x_2(0) = x_2^0 >0,\, \,  x_3(0)= x_3^0>0, \quad \text{and} \quad x_4(0) =x_4^0>0. $$

We refer the interested reader to \cite{CaiLi} for the full discussion of the derivation of \eqref{det eqn 1}. Due to biological interpretations, we are only concerned with non negative solutions in this work.
 A detailed description of parameters used in the model with units per day are given in Table~\ref{table}. 
 
\begin{center}
\begin{table}[ht]
\caption{List of parameters used in the model.}
\label{table}
\begin{tabular}{|c | l|}
 \hline
   Parameters & Description of the parameters \\ [0.5ex] 
 \hline\hline
  $\mu$ & Mortality rate of the host  \\ 
 \hline
  $\eta$ & Mortality rate of the vector   \\
 \hline
  $\phi$ &  Recovery rate of infected host  \\
 \hline
  $\beta_1$ & Direct transmission rate from an infected host to susceptible host   \\
 \hline
  $\beta_2$ & Indirect transmission rate from an infected vector to a susceptible vector  \\ 
  \hline
  $\beta$ &  The transition rate from infected host to susceptible vector  \\
  \hline
  $b_1$ &  The recruitment rate of the host \\
  \hline
  $b_2$ & The recruitment rate of the vector   \\ [1ex] 
 \hline
\end{tabular}
\end{table}
\end{center}

One way to see what will happen to the population eventually is to explore when the system is at equilibrium.  Let ${\bf x}=(x_1,x_2,x_3,x_4)^T$ denotes a column vector in the Euclidean space $\mathbb{R}^4$.
 It is easily seen that $E_0=(\frac{b_1}{\mu},0,\frac{b_2}{\eta},0)^T$ is always an equilibrium point of system \eqref{det eqn 1}. We note that $E_0=(\frac{b_1}{\mu},0,\frac{b_2}{\eta},0)^T$ represents the disease-free equilibrium or state. The disease-free equilibrium $E_0$ is the case where the pathogen has suffered extinction and, in the long run, everyone in the population is susceptible. It turns out that \eqref{det eqn 1} may have one more positive equilibrium point depending on the range of its parameters. Indeed, for convenience, we first introduce some quantities. Define
\begin{equation}\label{def-R_0} \mathcal{R}_0 =\frac{\beta_1 b_1}{\mu ( \mu +\phi)} + \frac{\beta \beta_2 b_1 b_2}{\eta^2 \mu (\phi+\mu)}.
 \end{equation}
 The constant $\mathcal{R}_0$ will be referred to as  the basic reproductive number of the model and it represents the expected number of secondary infections by a single infectious individual over a duration of time in a fully susceptible population \cite{CaiLi}. When $\mathcal{R}_0\leq 1$, it is well known that the disease-free  equilibrium $E_0$ is the only non-negative equilibrium of \eqref{det eqn 1}, in which case it is stable. However, when $\mathcal{R}_0>1$, \eqref{det eqn 1}  admits one more positive equilibrium $E_1=(x_1^{**},x_2^{**},x_3^{**},x_4^{**})^T$, where
  \begin{equation}\label{equili1}
x_3^{**}=\frac{b_2}{\eta +\beta x_2^{**}}, \ \ x_4^{**}=\frac{\beta b_2}{\eta}  \frac{x_2^{**}}{\eta +\beta x_2^{**}}, \ \
x_1^{**}=\frac{\eta(\mu +\phi)(\eta +\beta x_2^{**})}{\beta _1 \eta (\eta +\beta x_2^{**})+\beta \beta_2 b_2},
\end{equation}
and $x_2^{**}$ is the positive solution of the equation
\begin{equation}\label{equili2}
k_2 (x_2^{**})^2+k_1 x_2^{**} +k_0 = 0,
\end{equation}
with
\begin{eqnarray}\label{k0}
k_0&=& -\mu \eta ^2 (\mu +\phi)(\mathcal{R}_0 -1 ), \ \ \ \  k_2 = \beta \eta \beta_1 \mu\\[0.5ex]
k_1&= &\phi \beta \eta \mu + \eta^2 \beta_1 \mu + \beta b_2 \beta_2 \mu + \beta \eta \mu^2-\beta b_1 \eta \beta_1.
\end{eqnarray}
 The positive equilibrium $E_1$ will be referred to as the endemic-equilibrium of \eqref{det eqn 1}. The endemic equilibrium $E_1$ is the state where the disease cannot be totally eradicated but remains in the population.

The local and global dynamics of the solution of system \eqref{det eqn 1} is completely determined by the value of $\mathcal{R}_0.$ That is, if $\mathcal{R}_0 < 1$ the disease-free equilibrium point $E_0$ is both locally and globally asymptotically stable. Similarly, if $\mathcal{R}_0 > 1,$ then the endemic equilibrium point $E_1$ is both locally and globally asymptotically stable (see \cite{CaiLi}).

One limitation of the above model is that, it doesn't consider spatial migration of the host and vector population which is considered as a key factor when developing accurate predictive models of the spread of infections. In this paper, we include the population diffusion into the vector-host model described by system \eqref{det eqn 1}. To that end, the total population of the host and vector at location $y\in\R$ and time $t\ge0$ is divided into two compartments each. Let $x_1(t,y)$ and $x_2(t,y)$ denote the densities of susceptible and infected hosts. Similarly, let $x_3(t,y)$ and $x_4(t,y)$ denote the densities of susceptible and infected vectors. Thus by incorporating diffusion on both the host and vector populations, we have the following parabolic system of partial differential equations
\begin{equation}\label{pde1}
\begin{cases}
\frac{\partial x_1}{\partial t} =D_h \Delta_y x_1 + b_1 -\mu x_1 -\beta _2  x_1 x_4-\beta_1 x_1 x_2 + \phi x_2, & y\in\R, t>0, \cr
\frac{\partial x_2}{\partial t}=D_h \Delta_y x_2 + \beta_1 x_1 x_2 +\beta_2  x_1 x_4-(\mu + \phi)x_2, & y\in\R,\ t>0, \cr
\frac{\partial x_3}{\partial t}=D_v\Delta_y x_3 + b_2 -\eta x_3-\beta x_3 x_2, & y\in\R, \ t>0, \cr
\frac{\partial x_4}{\partial t}= D_v \Delta_y x_4 + \beta x_3 x_2-\eta x_4, & y\in\R, \ t>0,
\end{cases}
\end{equation}
where $\Delta u$ denotes the Laplace operator and  $D_v$ and $D_h$ correspond to the diffusion rates of the hosts $(x_1,x_2)^T$ and the vectors $(x_3,x_4)^T$, respectively. For convenience, we shall denote by ${\bf x}(t,y)=(x_1(t,y),x_2(t,y),x_3(t,y),x_4(t,y))^{T}$ the solutions of \eqref{pde1}.

In the biological context, it is important to analyze the epidemic wave which is described by traveling wave solutions propagating with a certain speed. The goal of the present work is to study the existence of traveling wave solutions of \eqref{pde1} (See Definitions \ref{def1} and \ref{def2} below for definition of traveling wave solutions).

 Many physical phenomena that arise in real world are a result of a wavelike event. 
Indeed, almost any film of a developing embryo is characterized by a wavelike
event that appear after fertilization. There are, for instance, both chemical and mechanical waves which propagate on the
surface of many vertebrate eggs. In addition, we can expect wave phenomena in interacting population
models where spatial effects are important. In particular, in the progressing wave of an epidemic. They arise in many areas of science including but not limited to combustion that may occur as a result of a chemical reaction, in mechanical deformation, in electrical signal and so on \cite[p.437]{Murray2002}.
Traveling waves are waves that move in a particular direction with a constant speed of propagation while retaining a fixed shape. The investigation of traveling wave solutions to nonlinear PDEs plays a central role in the modeling of nonlinear phenomena. The existence of such traveling waves is usually a consequence of the coupling of various effects
such as diffusion or chemotaxis or convection.
They have been used to model the spread of pest outbreaks, traveling waves of chemical concentration,
colonization of space by a population, spatial spread of epidemics and so on \cite[p.418]{Murray2002}. Furthermore, traveling wave solutions are used to describe the invasion of the disease free equilibrium by the endemic equilibrium with a constant speed.
Finally, a fish moves forward itself through water by a sequence of traveling waves which progress
down the fish's body from head to tail \cite[p.422]{Murray2002}. 

There are several works on traveling wave solutions of diffusive-reaction epidemic systems \cite{Ge_2009-1, Ge_2009-2, WaZh2006, Wu_Zo2001, Wu_Zou1998}. In \cite{Wu_Zo2001, Wu_Zou1998}, Wu and Zou studied the existence of traveling wave fronts for delayed reaction-diffusion systems with reaction terms satisfying the so called quasi-monotonicity or exponential quasi-monotonicity conditions.  Ge et al. \cite{Ge_2009-1, Ge_2009-2} used the iteration technique developed in \cite{Wu_Zo2001} to investigate the existence of traveling wave solutions for two-species predator-prey system with diffusion terms and stage structure, respectively.   Huang et al. \cite{Hang_Zou20016} employed the Schauder's fixed point theorem  to investigate the existence of traveling wave solutions of a class of delayed reaction diffusion systems with two equations.  Sazonov et al. \cite{Sazonov_2008, Sazonov2011}  studied problems of traveling waves in an SIR model. We refer the reader to \cite{Li2006, Yang_2011, Zhao} for more studies on the traveling wave solutions of epidemic-models.  It is important to point out that the mathematical techniques developed in their work cannot directly be applied to \eqref{pde1}.

\subsection*{Main Results}

 We state our main results in the following. We first introduce some definitions.

\begin{defin}\label{def1}
A positive bounded classical  solution ${\bf x}(t,y)$ of \eqref{pde1} is a traveling wave solution  with speed $c\in\R$ if it is non-constant and is of the form
$$
 {\bf x}(t,y)={\bf x}(y+ct),\quad \forall\ y\in\R, t\in\R.
$$
 A traveling wave solution  $
 {\bf x}(t,y)={\bf x}(y+ct)$ with speed $c$ is said to connect $ E_0=(\frac{b_1}{\mu},0,\frac{b_2}{\eta},0)^{T}$ at one end if it satisfies
 \begin{equation*}
  \lim_{y\to-\infty}{\bf x}(y)=E_0,
 \end{equation*}
 where  $ E_0=(\frac{b_1}{\mu},0,\frac{b_2}{\eta},0)^{T}$ is the disease-free equilibrium.

\end{defin}

 \begin{defin}\label{def2}
 Suppose that $\mathcal{R}_0>1$ and let $E_1=(x_1^{**},x_2^{**},x_3^{**},x_4^{**})^T$  denotes the endemic equilibrium solution of \eqref{pde1}. A traveling wave solution  $ {\bf x}(t,y)={\bf x}(y+ct)$ with speed $c$ is called a transition front connecting $E_0$ and $E_1$ if it satisfies
  \begin{equation*}
  \left(\lim_{y\to-\infty}{\bf x}(y),\lim_{y\to\infty}{\bf x}(y)\right)^T=(E_0,E_1)^T.
 \end{equation*}
 \end{defin}

 Our result on the existence of traveling wave solutions reads as follows.

 \medskip

 \begin{tm}\label{main-tm-02}
  Assume that $\mathcal{R}_0>1$.  There is a positive constant $c^*>0$ such that the following hold.  If $\phi\leq \frac{\beta_1b_1}{\mu}$ holds then :
  \begin{itemize}
  \item[(i)]For every $c> c^*$, \eqref{pde1} has a transition front solution ${\bf x}(t,y)={\bf x}(y+ct)$ with speed $c$ connecting $E_0$ and $E_1$.
  \item[(ii)] System \eqref{pde1} has a traveling wave solution ${\bf x}(t,y)={\bf x}(y+c^*t)$ with speed $c^*$ connecting $E_0$ and $E_1$.
  \end{itemize}

\end{tm}

\medskip

 \begin{rk}\label{main-rk1}
 \begin{itemize}
 \iffalse
 \item[(i)] It is worth mentioning that the existence of $c^*$ is only subject to the fact that $\mathcal{R}_0>1$. Hence it would of great mathematical interest to study the existence of traveling solutions of \eqref{pde1} connecting $E_0$ and $E_1$ when $\mathcal{R}_0>1$ and  $\phi>\frac{\beta_1b_1}{\mu}$ . It would also be interesting to study the existence of traveling waves of \eqref{pde1} when the hosts have different diffusion rates and/or when the vectors have different diffusion rates.  We plan to continuous working on these questions in our future work.
 \fi
 \item[(i)] As we will see below, when $\mathcal{R}_0<1$, \eqref{pde1} has no non-trivial traveling wave solution.

 \iffalse
 \item[(ii)] By Theorem \ref{main-tm-02} it would be interesting to know whether \eqref{pde1} has a minimal wave speed when  $\mathcal{R}_0>1$ and $\phi> \frac{\beta_1b_1}{\mu}$. This question is related to the non-existence of traveling wave solutions. In this direction
 \fi
 \item [(ii)] Our next result, specifically Theorem~\ref{main-tm-2}, indicates that $c^*$ is the minimal positive wave speed.
 \end{itemize}
 \end{rk}

 Our results on the non-existence of traveling wave solutions read as follows.

 \begin{tm}\label{main-tm-1} Assume that $\mathcal{R}_0<1$. Then for every $c\in\R$, \eqref{pde1} has no traveling wave  solution ${\bf x}(t,y)={\bf x}(y+ct)$ connecting $E_0$ at one end.
\end{tm}

\begin{tm}\label{main-tm-2}
  Assume that $\mathcal{R}_0>1$ and let $c^*$ be given by Theorem \ref{main-tm-02}. Then the following hold.
\begin{itemize}
\item[(i)] System \eqref{pde1} has no traveling wave  solution ${\bf x}(t,y)={\bf x}(y+ct)$ with speed  $|c|<c^*$ connecting $E_0$ at one end.

\item[(ii)]   If in addition  $D_h=D_v$ holds, then  \eqref{pde1} has no traveling wave  solution ${\bf x}(t,y)={\bf x}(y+ct)$ with speed  $c<c^*$ connecting $E_0$ at one end.
\end{itemize}
\end{tm}

\begin{rk}\label{main-rk2}
\begin{itemize}
\item[(i)] Theorem \ref{main-tm-1} shows that \eqref{pde1} has no non-trivial traveling wave solution when $\mathcal{R}_0<1$. We refer the interested reader to \cite{CaiLi} for the study of dynamics of the diffusion free system of \eqref{pde1}.
\item[(ii)] Assume that $\mathcal{R}_0>1$.  If  $\phi\leq \frac{\beta_1b_1}{\mu}$, it follows from Theorems \ref{main-tm-02} and \ref{main-tm-2} (i) that $c^*$ is the minimal non-negative  wave speed of \eqref{pde1}. And if  in addition $D_h=D_v$ holds, then $c^*$ is the minimal wave speed of \eqref{pde1}.
\item[(iii)] It is natural to ask whether $c^*$ is always the minimal wave speed when $\mathcal{R}_0>1$ and $\phi>\frac{\beta_1b_1}{\mu}$. This question is also related to the existence of traveling waves when $\phi>\frac{\beta_1b_1}{\mu}$.
\end{itemize}
\end{rk}

The rest of the paper is organized as follows. In section 2, we establish some preliminary results to be used in the subsequent sections. It is here that we define the positive constant $c^*$ when $\mathcal{R}_0>1$. In section 3, we construct some super-sub solutions that are used in the proof of the existence of traveling wave solutions. Section 4 is devoted to the proof of Theorem \ref{main-tm-02} while section 5  is devoted for the proof of Theorems \ref{main-tm-1} and \ref{main-tm-2}. In the last section 6, we conclude this work with some numerical simulations to illustrate our theoretical results. Explicit values of $c^*$ and $\mathcal{R}_0$ are computed for a range of given parameters. The numerical simulations suggest that the traveling wave solutions are not monotone.

 \section{Preliminaries}\label{Sec1}

 In the current section, we present some preliminary results that will be needed for the subsequent sections.     For convenience, we introduce the following.
$$
l_0:=\mu+\phi-\frac{\beta_1b_1}{\mu} \quad \text{and}\quad l_1:=\frac{\beta\beta_2b_1b_2}{\mu\eta}.
$$
Define for ${\bf x}=(x_1,x_2,x_3,x_4)^T\in\R^4$,  $F({\bf x})=(F_1({\bf x}),F_2({\bf x}),F_3({\bf x}),F_4({\bf x}))$  by
\begin{equation}\label{F-def}
F({\bf x})=\left(\begin{array}{c}
b_1-(\mu+\beta_2x_4+\beta_1x_2)x_1+\phi x_2\cr
(\beta_1x_1-(\phi+\mu))x_2+\beta_2x_1x_4\cr
b_2-(\eta+\beta x_2 )x_3\cr
\beta x_2x_3-\eta x_4\end{array}\right).
\end{equation}
It is more convenient to write \eqref{pde1} in the form
 \begin{align}\label{s-eq1}
 \partial_t {\bf x}(t,y)=\text{diag}(d_i\Delta_y x_{i} )+F({\bf x}(t,y)),\quad y\in\R, t>0,
 \end{align}
where $d_1=d_2=D_v$ and $d_3=d_4=D_h$.

We start first by linearizing \eqref{s-eq1} at $(\frac{b_1}{\mu},0,\frac{b_2}{\eta},0)^T$ and obtain

\begin{align}\label{s-eq2}
 \begin{cases}
 \frac{\partial x_1}{\partial t}=D_h\Delta_y x_1-\mu x_1-(\frac{\beta_1b_1}{\mu}-\phi)x_2 -\frac{\beta_2b_1}{\mu}x_4,& y\in\mathbb{R}, t>0,\cr
 \frac{\partial x_2}{\partial t}=D_h\Delta_y x_2 -l_0x_2+\frac{\beta_2b_1}{\mu}x_4,& y\in\mathbb{R}, t>0,\cr
 \frac{\partial x_3}{\partial t}=D_v\Delta_y x_3-\eta x_3-\frac{\beta b_2}{\eta}x_2,& y\in\mathbb{R}, t>0,\cr
 \frac{\partial x_4}{\partial t}=D_v\Delta_y x_4-\eta x_4+\frac{\beta b_2}{\eta} x_2,& y\in\mathbb{R}, t>0.
 \end{cases}
\end{align}

We observe that both $x_1$ and $x_3$ do not appear in the sub-system formed by the equations given by $x_2$ and $x_4$ in \eqref{s-eq2}. And both $x_2$ and $x_4$ determine uniquely $x_1$ and $x_3$. Hence, the dynamic of solutions of \eqref{s-eq2} is completely determined by those solutions of the sub-system formed by both $x_2$ and $x_4$. This justifies why we should focus on the equations given by $x_2$ and $x_4$ in \eqref{s-eq2},
\begin{align}\label{s-eq3}
 \begin{cases}
 \frac{\partial x_2}{\partial t}=D_h\Delta_y x_2 -l_0x_2+\frac{\beta_2b_1}{\mu}x_4,& y\in\mathbb{R}, t>0,\cr
 \frac{\partial x_4}{\partial t}=D_v\Delta_y x_4-\eta x_4+\frac{\beta b_2}{\eta} x_2,& y\in\mathbb{R}, t>0.
 \end{cases}
\end{align}
Suppose that \eqref{s-eq3} has a positive solution of the form $(x_2(t,y),x_4(t,y))=(k_2,k_4)e^{\lambda(y+ct)}$ for some $\lambda,c\in\R$ and positive real numbers $k_2$ and $k_4$. Then $\lambda,c,k_2$ and $k_4$ satisfy
\begin{equation}\label{s-eq4}
 \begin{cases}
 \lambda c k_2=(D_h\lambda^2-l_0)k_2+\frac{\beta_2b_1}{\mu}k_4,\cr
 \lambda ck_4=(D_v \lambda^2-\eta )k_4+\frac{\beta b_2}{\eta} k_2.
 \end{cases}
\end{equation}
Equivalently, \eqref{s-eq4} can be written in the form
\begin{equation}\label{s-eq5}
\left[\begin{array}{cc}
D_h\lambda^2-l_0 & \frac{\beta_2b_1}{\mu}\cr
\frac{\beta b_2}{\eta} & D_v \lambda^2-\eta
\end{array}\right]\left(\begin{array}{c}
k_2\cr
k_4
\end{array}\right)=\lambda c\left(\begin{array}{c}
k_2\cr
k_4
\end{array}\right).
\end{equation}
Thus, if \eqref{s-eq4} has a solution $(k_2,k_4)\in[\R^+]^2$ for some $\lambda\in\R$ and $c\in\R$, we must have that $\lambda c$ is an eigenvalue of the matrix
$$
\mathcal{M}(\lambda):=\left[\begin{array}{cc}
D_h\lambda^2-l_0& \frac{\beta_2b_1}{\mu}\cr
\frac{\beta b_2}{\eta} & D_v \lambda^2-\eta
\end{array}\right]
$$
and $(k_2,k_4)\in[\R^+]^2$ is an eigenvector corresponding to $\lambda c$. Observe that the off diagonal entries of the matrix $\mathcal{M}(\lambda)$ are positive real numbers. Thus by Perron-Frobenius's theorem its dominant eigenvalue, which will be denoted by $\alpha_{\max}(\lambda)$, is a real number. We are mainly interested in the situation for which the dominant eigenvalue  $\alpha_{\max}(\lambda)$ of the matrix $\mathcal{M}(\lambda)$ is positive for every $\lambda\in\R$. Note that such hypothesis, if exists, implies that the origin is unstable for the flow of solutions generated by solutions of \eqref{s-eq3} on $\R^2$.

The characteristic polynomial $p_\lambda(\alpha)$ of the matrix $\mathcal{M}(\lambda)$ is
$$
p_{\lambda}(\alpha)=\text{det}(\mathcal{M}(\lambda)-\alpha \mathcal{I})=(m_1(\lambda)-\alpha)(m_2(\lambda)-\alpha)-l_1,
$$
where $$m_1(\lambda)=D_h\lambda^2-l_0 \quad \text{ and}\quad  m_2(\lambda)=D_v \lambda^2-\eta, $$  and $  \mathcal{I} $ denote the square identity matrix. The quadratic formula yield that 
 the two roots of the equation $p_{\lambda}(\alpha)=0$ are given for every $\lambda\in\R$, by
\begin{align}\label{s-eq7}
\alpha_{min}(\lambda)=&\frac{1}{2}\left(m_1(\lambda)+m_2(\lambda)-\sqrt{(m_1(\lambda)-m_2(\lambda))^2+4\frac{\beta\beta_2b_1b_2}{\eta\mu}}\right)\cr
=&\frac{1}{2}\left((D_h+D_v)\lambda^2-(\eta+l_0)-\sqrt{(\eta-l_0+(D_h-D_v)\lambda^2)^2+4l_1} \right),
\end{align}
and
\begin{align}\label{s-eq8}
\alpha_{max}(\lambda)=&\frac{1}{2}\left(m_1(\lambda)+m_2(\lambda)+\sqrt{(m_1(\lambda)-m_2(\lambda))^2+4\frac{\beta\beta_2b_1b_2}{\eta\mu}}\right)\cr
=&\frac{1}{2}\left((D_h+D_v)\lambda^2-(\eta+l_0)+\sqrt{(\eta-l_0+(D_h-D_v)\lambda^2)^2+4l_1} \right).
\end{align}
In particular,
$$
\alpha_{\max}(0)=\frac{1}{2}\left(-(\eta+l_0)+\sqrt{(\eta-l_0)^2+4l_1} \right).$$

Observe that $\frac{\beta_1b_1}{\mu}>\mu+\phi$ implies that $\alpha_{\max}(0)>0$.  %(its proof is given below in the proof of Lemma \ref{lemma2}, case 1). 
Note that the eigenspace, say $\mathcal{E}_{\alpha}$, associated with the eigenvalue $\alpha\in\{\alpha_{\min}(\lambda),\alpha_{\max}(\lambda)\}$ is given by
\begin{equation}\label{s-eq10}
\mathcal{E}_{\alpha}=\text{span}
\left\{\left(\begin{array}{c}
1\cr
\frac{(\alpha-m_1(\lambda))\mu}{\beta_2b_1}
\end{array}\right)\right\}.
\end{equation}

The following lemma collects few properties of the function $\alpha_{\max}(\lambda)$.
\begin{lem}\label{lemma0}
Consider the function $\R\ni\lambda \mapsto \alpha_{\max}(\lambda)$, where $\alpha_{\max}(\lambda)$ is given by \eqref{s-eq8}. The following holds.
\begin{itemize}
\item[(i)] The function $\alpha_{\max}(\lambda)$   is an even function and strictly convex.
\item[(ii)] It holds that \begin{equation}\label{s-eq9}
\alpha_{\min}(\lambda)<\min\{m_1(\lambda),m_2(\lambda)\}\leq \max\{m_1(\lambda),m_2(\lambda) \}<\alpha_{max}(\lambda),\quad \forall\lambda\in\R.
\end{equation}
\item[(iii)]The function  $\alpha_{\max}(\lambda)$ is strictly increasing on the half interval $[0,\infty)$, hence $\alpha_{\max}(0)< \alpha_{\max}(\lambda)$ for every $\lambda\in\R\setminus\{0\}$.
\end{itemize}
\end{lem}
\begin{proof}
$(i)$ The fact that $\alpha_{\max}(\lambda)$  is an even function easily follows from its expression. It is easily seen that each of the functions
$$
\alpha_{max,1}(\lambda)=(D_h+D_v)\lambda^2-(\eta+l_0) \quad \text{and}\quad \alpha_{max,2}(\lambda)=\sqrt{(\eta-l_0+(D_h-D_v)\lambda^2)^2+4l_1}
$$
are convex on $\R^+$ with $\alpha_{max,1}(\lambda)$ strictly convex. Hence, we conclude that $\alpha_{max}(\lambda)=\frac{1}{2}(\alpha_{max,1}(\lambda)+\alpha_{max,2}(\lambda))$ is also strictly convex.

$(ii)$ We note from the expression of $p_{\lambda}(\alpha)$ that
\begin{equation*}
(m_1(\lambda)-\alpha)(m_2(\lambda)-\alpha)=l_1>0 \quad \text{for}\ \alpha\in\{\alpha_{\min}(\lambda),\alpha_{\max}(\lambda)\},\quad \lambda\in\R.
\end{equation*}
Hence, since $\alpha_{\min}(\lambda)+\alpha_{\max}(\lambda)=m_1(\lambda)+m_2(\lambda)$, we conclude that \eqref{s-eq9} holds.

$(iii)$ The function $\R\ni\lambda \mapsto \alpha_{\max}(\lambda)$ is of class $C^{\infty}$ with
\begin{equation*}
\lambda\alpha_{\max}'(\lambda)=\lambda^2\left( (D_h+D_v)+\frac{(D_h-D_v)(m_1(\lambda)-m_2(\lambda))}{\sqrt{(m_1(\lambda)-m_2(\lambda))^2+4\frac{\beta\beta_2b_1b_2}{\eta\mu}}} \right)>0,\quad \forall\lambda\in\R\setminus\{0\}.
\end{equation*}
Hence, $(iii)$ follows.
\end{proof}
 For every $\lambda\neq 0$, define
\begin{equation}\label{s-eq6} c_{\lambda}=\frac{\alpha_{\max}(\lambda)}{\lambda}.\end{equation}
 We have the following result.
\begin{lem}\label{lemma1} Assume that $\alpha_{\max}(0)>0$.
\begin{itemize}
\item[(i)]The function $(0,\infty)\ni\lambda\mapsto c_{\lambda} $ has a  minimum  value $c^*$, which is achieved at some $\lambda^*>0$. Moreover, the positive number $\lambda^*$ is uniquely determined and for every $c>c^*$ the equation $c=c_{\lambda}$ has exactly two positive roots $0<\lambda_{\min}(c)<\lambda^*<\lambda_{\max}(c)$.

\item[(ii)]The restriction of $c_{\lambda}$ on $(-\infty,0)$ has a  maximum  value given by $-c^*$, where $c^*$ is given by $(i)$. Moreover, for every $c<-c^*$ the equation $c=c_{\lambda}$ has exactly two negative roots given by $-\lambda_{\max}(-c)$ and $-\lambda_{\min}(-c)$.
\end{itemize}

\end{lem}

\begin{proof} It is clear that $(ii)$ follows from $(i)$ since the function $\alpha_{\max}(\lambda)$ is an even function. So, we only prove that $(i)$ holds. For, since $\alpha_{\max}(0)>0$ holds, then by Lemma \ref{lemma0} we have  that $\alpha_{\max}(\lambda)> \alpha_{\max}(0)>0$ for every $\lambda>0$. Hence $c_{\lambda}>0$ for every $\lambda>0$. Observe that
\begin{equation}\label{sa-0}
c_\lambda\geq \frac{1}{2\lambda}(m_1(\lambda)+m_2(\lambda))=\frac{1}{2\lambda}((D_h+D_v)\lambda^2-(l_0+\eta))\to+\infty\quad \text{as}\ \lambda\to\infty.
\end{equation}
Note also  that
\begin{align}\label{sa-1}
\lim_{\lambda\to0^+}c_\lambda=\infty,
\end{align}
since $\lim_{\lambda\to0^+}\alpha_{\max}(\lambda)=\alpha_{\max}(0)>0$. Hence the existence of $c^*$ follows  due to the continuity of the function $\lambda\mapsto c_{\lambda}$. Let $\lambda^*>0$ such that $c^*=c_{\lambda^*}$, whose existence is guaranteed by the intermediate value theorem. To prove that $\lambda^*$ is uniquely determined, it is enough to show that  the equation $ c_{\lambda}=c$  has exactly two positive roots for every $c>c^*$. Indeed, let  $c>c^*$, it follows from \eqref{sa-0} and \eqref{sa-1},  and the intermediate value theorem that there exist $0<\lambda_{\min}(c)<\lambda^*<\lambda_{\max}(c)$ such that
\begin{equation}\label{sa-2}
c=c_{\lambda_{\max}(c)}=c_{\lambda_{\min}(c)}.
\end{equation}
This implies that the straight line (in $\lambda y$ plane) with equation $y=c\lambda$ and the graph of the function $\alpha_{\max}(\lambda)$ intersect  at two different points $(\lambda_{\min}(c),c\lambda_{\min}(c))$ and $(\lambda_{\max}(c),c\lambda_{\max}(c))$. But by Lemma \ref{lemma0}, the function $\alpha_{\max}(\lambda)$ is strictly convex, hence we deduce that $\lambda_{\min}(c)$ and $\lambda_{\max}(c)$ are the only solution of \eqref{sa-2}. Which completes the proof of the lemma.

\end{proof}

Next, we define
\begin{equation}\label{s-eq12}
c^*=c_{\lambda^*}=\min_{\lambda>0}c_{\lambda}=\min_{\lambda>0}\frac{\alpha_{\max}(\lambda)}{\lambda}.
\end{equation}
where $(c^*,\lambda^*)$ are given by the previous lemma.

\begin{rk}\label{remark1} Assume that $\alpha_{\max}(0)>0$.  By Lemma \ref{lemma1} the function $(0,\lambda^*)\ni\lambda \mapsto c_{\lambda}$ is strictly decreasing. Let $0<\lambda<\lambda^*$ be fixed. The following hold.
\begin{itemize}
\item[(i)] By \eqref{s-eq9}, $\alpha_{\max}(\lambda)>m_1(\lambda)$. Hence by \eqref{s-eq10}, the vector
$$
\left(\begin{array}{c}
k_{2,\lambda}\cr
k_{4,\lambda}\end{array}\right):=\left(\begin{array}{c}
1\cr
\frac{(\alpha_{\max}(\lambda)-m_1(\lambda))\mu}{\beta_2b_1}
\end{array}\right)
$$
is an eigenvector of $\mathcal{M}(\lambda)$, with positive coordinates, associated  to $\alpha_{\max}(\lambda)$. In particular $(\lambda,c_\lambda,k_{2,\lambda},k_{4,\lambda})$ is a solution of \eqref{s-eq4}.
\item[(ii)] For every $\kappa\in(0,\lambda^*-\lambda)$, it holds that $ c_{\lambda+\kappa}<c_{\lambda}$, since $0<\lambda<\lambda+\kappa<\lambda^*$. Thus
$$
\mathcal{M}(\lambda+\kappa)\left(\begin{array}{c}
k_{2,\lambda+\kappa}\cr
k_{4,\lambda+\kappa}\end{array}\right)=(\lambda+\kappa)c_{\lambda+\kappa}\left(\begin{array}{c}
k_{2,\lambda+\kappa}\cr
k_{4,\lambda+\kappa}\end{array}\right)<(\lambda+\kappa)c_\lambda\left(\begin{array}{c}
k_{2,\lambda+\kappa}\cr
k_{4,\lambda+\kappa}\end{array}\right),
$$
for every $\kappa\in(0,\lambda^*-\lambda)$.
\item[(iii)] Suppose that $D_h=D_v=D$. Then $\alpha_{\max}(\lambda)=D\lambda^2+\alpha_{\max}(0)$ for every $\lambda$. In this case, we have that  $\lambda^*=\sqrt{\frac{\alpha_{\max}(0)}{D}}$ and $c^*=2\sqrt{D\alpha_{\max}(0)}$.
\end{itemize}
\end{rk}

\section{Super-sub solutions}
In this section we construct super-sub solutions to be used in the next section. Throughout this section we suppose  that $\alpha_{\max}(0)>0 $  so that the positive constant $c^*$ is defined by Lemma \ref{lemma1}. Let $c>c^*$ be given and let $\lambda=\lambda_{\min}(c)<\lambda^*$ so that $c=c_{\lambda}$.  For every $\kappa\in[0,\lambda^*-\lambda)$ let
\begin{equation}\label{s-eq13}
\overline{x}_{1,\lambda+\kappa}(y)=\frac{b_1}{\mu},\quad \forall y\in\R,
\end{equation}
\begin{equation}\label{s-eq14}
\overline{x}_{2,\lambda+\kappa}(y)=k_{2,\lambda+\kappa}e^{(\lambda+\kappa) y}\quad \forall y\in\R,
\end{equation}
\begin{equation}\label{s-eq15}
\overline{x}_{3,\lambda+\kappa}(y)=\frac{b_2}{\eta},\quad \forall y\in\R,
\end{equation}
\begin{equation}\label{s-eq16}
\overline{x}_{4,\lambda+\kappa}(y)=k_{4,\lambda+\kappa}e^{(\lambda+\kappa) y},\quad \forall y\in\R,
\end{equation}
and define
\begin{equation}\label{sub-sol-eq}
\overline{\bf x}_{\lambda}(y):=(\overline{x}_{1,\lambda}(y),\overline{x}_{2,\lambda}(y),\overline{x}_{3,\lambda}(y),\overline{x}_{4,\lambda}(y))^T, \quad \forall \ y\in\R.
\end{equation}
The following lemma follows from the definition of $c_\lambda$ and Remark  \ref{remark1} (i).

\begin{lem}\label{lemma3} Let $\overline{\bf x}_{2,\lambda+\kappa}(y)$ and $\overline{\bf x}_{4,\lambda+\kappa}(y)$ be defined as in the above for every $\kappa\in[0,\lambda^*-\lambda)$. Then
\begin{align*}
\begin{cases}
  0=D_h\Delta_y\overline{x}_{2,\lambda+\kappa}-c_\lambda \frac{d}{dy}\overline{x}_{2,\lambda+\kappa}-l_0\overline{x}_{2,\lambda+\kappa}+\frac{\beta_2b_1}{\mu}\overline{x}_{4,\lambda+\kappa},\cr
 0=D_v\Delta_y\overline{x}_{4,\lambda+\kappa}-c_\lambda \frac{d}{dy}\overline{x}_{4,\lambda+\kappa}-\eta\overline{x}_{4,\lambda+\kappa}+\frac{\beta b_2}{\eta}\overline{x}_{2,\lambda+\kappa}.
 \end{cases}
\end{align*}
\end{lem}

Next, given $A,B>1$ and $0<\kappa,\tilde{\lambda}\ll 1$, we define
\begin{equation}\label{s-eq17}
\underline{x}_{1,\lambda}(y)=(\overline{x}_{1,\lambda}(y)-Ae^{\tilde{\lambda}y})_+=\left(\frac{b_1}{\mu}-Ae^{\tilde{\lambda}y}\right)_+,\quad \forall y\in\R,
\end{equation}
\begin{equation}\label{s-eq18}
\underline{x}_{2,\lambda}(y)=(\overline{x}_{2,\lambda}(y)-B\overline{x}_{2,\lambda+\kappa}(y))_+=(k_{2,\lambda}-Bk_{2,\lambda+\kappa}e^{\kappa y})_+e^{\lambda y}\quad \forall y\in\R,
\end{equation}
\begin{equation}\label{s-eq19}
\underline{x}_{3,\lambda}(y)=(\overline{x}_{3,\lambda}(y)-Ae^{\tilde{\lambda}y})_+=\left(\frac{b_2}{\eta}-Ae^{\tilde{\lambda} y}\right)_+,\quad \forall y\in\R,
\end{equation}
\begin{equation}\label{s-eq20}
\underline{x}_{4,\lambda}(y)=(\overline{x}_{4,\lambda}(y)-B\overline{x}_{4,\lambda+\kappa}(y))_+=(k_{4,\lambda}-Bk_{4,\lambda+\kappa}e^{\kappa y})_+e^{\lambda y},\quad \forall y\in\R,
\end{equation}
and define
\begin{equation}\label{supper-sol-eq}
 \underline{\bf x}_{\lambda}(y):=(\underline{x}_{1,\lambda}(y),\underline{x}_{2,\lambda}(y),\underline{x}_{3,\lambda}(y),\underline{x}_{4,\lambda}(y))^T,\quad \forall\ y\in\R
\end{equation}
where $m_+=\max\{0, m\}$. The following hold.

\begin{lem}\label{lemma4}
For every $\tilde{\lambda}$ and $A$ satisfying \begin{equation}\label{s-eq-s2}
0<\tilde{\lambda}\leq \min\left\{\lambda,\frac{c_\lambda}{D_h+D_v}\right\}\quad\text{and}\quad A\geq \max\left\{ 1,\frac{b_1(\beta_2k_{4,\lambda}+\beta_1k_{2,\lambda})}{\mu^2},\frac{\eta\mu}{b_2k_{2,\lambda}}\right\},\end{equation} it holds that

\begin{equation}\label{s-eq21}
 0\leq D_h\Delta_y\underline{x}_{1,\lambda}-c_\lambda \frac{d}{dy}\underline{x}_{1,\lambda}+b_1-(\mu+\beta_2\overline{x}_{4,\lambda}+\beta_1\overline{x}_{2,\lambda})\underline{x}_{1,\lambda},
\end{equation}
and
\begin{equation}\label{s-eq22}
 0\leq D_v\Delta_y\underline{x}_{3,\lambda}-c_\lambda \frac{d}{dy}\underline{x}_{3,\lambda}+b_2-(\eta+\beta\overline{x}_{2,\lambda})\underline{x}_{3,\lambda}.
\end{equation}

\end{lem}
\begin{proof}
We will only prove \eqref{s-eq21}, since the proof of \eqref{s-eq22} follows similar arguments.  We check \eqref{s-eq21} when $\underline{x}_{1,\lambda}(y)>0$, since the inequality holds trivially otherwise. For $y\in\R$ such that $\underline{x}_{1,\lambda}(y)>0$, we have $y< 0$ and
\begin{align*}
&D_h\Delta_y\underline{x}_{1,\lambda}- c_\lambda\frac{d}{dy}\underline{x}_{1,\lambda}+b_1-(\mu+\beta_2\overline{x}_{4,\lambda}+\beta_1\overline{x}_{2,\lambda})\underline{x}_{1,\lambda}\cr
=& -D_hA\tilde{\lambda}^2e^{\tilde{\lambda}y}-A c_\lambda\tilde{\lambda}e^{\tilde{\lambda}y}-(\mu+(\beta_2k_{4,\lambda}+\beta_1k_{2,\lambda})e^{\lambda y})\left(\frac{b_1}{\mu}-Ae^{\tilde{\lambda}y}\right)+b_1\cr
=&A(\mu+\tilde{\lambda}( c_\lambda-D_h\tilde{\lambda}))e^{\tilde{\lambda}y}-(\beta_2k_{4,\lambda}+\beta_1k_{2,\lambda})\left(\frac{b_1}{\mu}-Ae^{\tilde{\lambda}y}\right)e^{\lambda y}\cr
\geq &\left( A(\mu+\tilde{\lambda}( c_\lambda-D_h\tilde{\lambda}))-\frac{b_1}{\mu}(\beta_2k_{4,\lambda}+\beta_1k_{2,\lambda})e^{(\lambda-\tilde{\lambda})y}\right)e^{\tilde{\lambda} y}\cr
\geq& \left( \mu A-\frac{b_1}{\mu}(\beta_2k_{4,\lambda}+\beta_1k_{2,\lambda})\right)e^{\tilde{\lambda}y}>0.
\end{align*}
Hence, \eqref{s-eq21} holds.
\end{proof}

 Let $\tilde{\lambda}$ and $A$ be fixed satisfying \eqref{s-eq-s2}. Then by \eqref{s-eq17} and \eqref{s-eq19}, it holds that
$$
\underline{x}_{1,\lambda}(y)=\overline{x}_{1,\lambda}-Ae^{\tilde{\lambda}y} \quad \text{and}\quad
\underline{x}_{3,\lambda}(y)=\overline{x}_{3,\lambda}-Ae^{\tilde{\lambda}y},\quad \forall y\leq -\frac{1}{\tilde{\lambda}}\max\left\{\ln\left(\frac{A}{\overline{x}_{1,\lambda}}\right),\ln\left(\frac{A}{\overline{x}_{3,\lambda}}\right)\right\}.
$$
Similarly, by \eqref{s-eq18} and \eqref{s-eq20},
$$
\underline{x}_{2,\lambda}(y)=\overline{x}_{2,\lambda}(y)-B\overline{x}_{2,\lambda+\kappa}(y),\quad \forall y\leq - \frac{1}{\kappa}\ln\left(\frac{Bk_{2,\lambda+\kappa}}{k_{2,\lambda}}\right),
$$
and
 $$
\underline{x}_{4,\lambda}(y)=\overline{x}_{4,\lambda}(y)-B\overline{x}_{4,\lambda+\kappa}(y),\quad \forall\ y\leq -\frac{1}{\kappa}\ln\left(\frac{Bk_{4,\lambda+\kappa}}{k_{4,\lambda}}\right).
 $$
Since
$$
\lim_{B\to\infty}\frac{1}{\kappa}\min\left\{\ln\left(\frac{Bk_{2,\lambda+\kappa}}{k_{2,\lambda}}\right),\ln\left(\frac{Bk_{4,\lambda+\kappa}}{k_{4,\lambda}}\right)\right\}=\infty,
$$
there is $B_0\gg 1$ such that
\begin{equation}\label{s-eq-s4}
\frac{1}{\kappa}\min\left\{\ln\left(\frac{Bk_{2,\lambda+\kappa}}{k_{2,\lambda}}\right),\ln\left(\frac{Bk_{4,\lambda+\kappa}}{k_{4,\lambda}}\right)\right\}>\frac{1}{\tilde{\lambda}} \max\left\{\ln\left(\frac{A}{\overline{x}_{1,\lambda}}\right),\ln\left(\frac{A}{\overline{x}_{3,\lambda}}\right)\right\} \quad \forall\ B\geq B_0.
\end{equation}

\begin{lem}\label{lemma5} Let $\tilde{\lambda}>0$ and $A>1$ be fixed satisfying \eqref{s-eq-s2}. Let $\kappa>0$ and $B>1$ satisfying
\begin{equation}\label{s-eq-s5}
0<\kappa<\min\{\tilde{\lambda},\lambda^*-\lambda\}\, \text{and}\, B\geq \max\left\{1,\frac{A(\beta_1k_{2,\lambda}+\beta_2k_{4,\lambda+\kappa})}{(\lambda+\kappa)(c_{\lambda}-c_{\lambda+\kappa})k_{2,\lambda+\kappa}},\frac{A\beta k_{2,\lambda}}{(\lambda+\kappa)(c_{\lambda}-c_{\lambda+\kappa})k_{4,\lambda+\kappa}}, B_0 \right\},
\end{equation}
where $B_0$ is given by \eqref{s-eq-s4}. Then
\begin{align}\label{s-eq23}
 \begin{cases}
 0\leq D_h\Delta_y\underline{x}_{2,\lambda}-c_\lambda\frac{d}{dy}\underline{x}_{2,\lambda}+(\beta_1\underline{x}_{1,\lambda}-(\phi+\mu))\underline{x}_{2,\lambda}+\beta_2\underline{x}_{1,\lambda}\underline{x}_{4,\lambda},\cr
 0\leq D_v\Delta_y\underline{x}_{4,\lambda}-c_\lambda\frac{d}{dy}\underline{x}_{4,\lambda} -\eta\underline{x}_{4,\lambda}+\beta\underline{x}_{2,\lambda}\underline{x}_{3,\lambda}.
 \end{cases}
 \end{align}

\end{lem}

\begin{proof} We only present the proof of the first inequality of \eqref{s-eq23} as the proof of the second inequality follows similar arguments. Observe  from the choice of the parameters that both $\underline{x}_{1,\lambda}(y)$ and $\underline{x}_{3,\lambda}(y)$ are positive.   We check \eqref{s-eq23} when $\underline{x}_{2,\lambda}(y)>0$, since the inequality holds trivially otherwise. For $y\in\R$ such that $\underline{x}_{2,\lambda}(y)>0$, we have $y<0$ and
\begin{align*}
&D_h\Delta_y\underline{x}_{2,\lambda}-c_\lambda\frac{d}{dy}\underline{x}_{2,\lambda}+(\beta_1\underline{x}_{1,\lambda}-(\phi+\mu))\underline{x}_{2,\lambda}+\beta_2\underline{x}_{1,\lambda}\underline{x}_{4,\lambda}\cr
=&D_h\left(\Delta_y\overline{x}_{2,\lambda}-B\Delta_y\overline{x}_{2,\lambda+\kappa}\right)-c_\lambda\left(\frac{d}{dy}\overline{x}_{2,\lambda}-B\frac{d}{dx}\overline{x}_{2,\lambda+\kappa}\right)\cr
&+(\beta_1\overline{x}_{1,\lambda}-(\phi+\mu)-\beta_1Ae^{\tilde{\lambda}y})\left(\overline{x}_{2,\lambda}-B\overline{x}_{2,\lambda+\kappa}\right)+\beta_2(\overline{x}_{1,\lambda}-Ae^{\tilde{\lambda}y})\left(\overline{x}_{4,\lambda}-B\overline{x}_{4,\lambda+\kappa}\right)\cr
=&\left\{D_h\Delta_y\overline{x}_{2,\lambda}-c_\lambda \frac{d}{dy}\overline{x}_{2,\lambda}-l_0\overline{x}_{2,\lambda}+\frac{\beta_2b_1}{\mu}\overline{x}_{4,\lambda}\right\}\cr
&-B\left\{D_h\Delta_y\overline{x}_{2,\lambda+\kappa}-c_{\lambda+\kappa} \frac{d}{dy}\overline{x}_{2,\lambda+\kappa}-l_0\overline{x}_{2,\lambda+\kappa}+\frac{\beta_2b_1}{\mu}\overline{x}_{4,\lambda+\kappa}\right\}\cr
&+B(c_\lambda-c_{\lambda+\kappa}) \frac{d}{dy}\overline{x}_{2,\lambda+\kappa}-A\left(\beta_1\overline{x}_{2,\lambda}+\beta_2\overline{x}_{4,\lambda}\right)e^{\tilde{\lambda}y} +BA\left(\beta_1\overline{x}_{2,\lambda+\kappa}+\beta_2\overline{x}_{4,\lambda+\kappa}\right)e^{\tilde{\lambda}y}.
\end{align*}
By Lemma \ref{lemma3}, the first two expressions in brackets are equal to zero, hence only the last term remains, which yields that
\begin{align*}
&D_h\Delta_y\underline{x}_{2,\lambda}-c_\lambda\frac{d}{dy}\underline{x}_{2,\lambda}+(\beta_1\underline{x}_{1,\lambda}-(\phi+\mu))\underline{x}_{2,\lambda}+\beta_2\underline{x}_{1,\lambda}\underline{x}_{4,\lambda}\cr
=&B(c_\lambda-c_{\lambda+\kappa}) \frac{d}{dy}\overline{x}_{2,\lambda+\kappa}-A\left(\beta_1\overline{x}_{2,\lambda}+\beta_2\overline{x}_{4,\lambda}\right)e^{\tilde{\lambda}y} +BA\left(\beta_1\overline{x}_{2,\lambda+\kappa}+\beta_2\overline{x}_{4,\lambda+\kappa}\right)e^{\tilde{\lambda}y}\cr
\geq & B(c_\lambda-c_{\lambda+\kappa}) \frac{d}{dy}\overline{x}_{2,\lambda+\kappa}-A\left(\beta_1\overline{x}_{2,\lambda}+\beta_2\overline{x}_{4,\lambda}\right)e^{\tilde{\lambda}y}\cr
=&\left((\lambda+\kappa)(c_{\lambda}-c_{\lambda+\kappa})Bk_{2,\lambda+\kappa}-A(\beta_1k_{2,\lambda}+\beta_2k_{4,\lambda+\kappa})e^{(\tilde{\lambda}-\kappa)y}\right)e^{\lambda y}\cr
\geq &\left((\lambda+\kappa)(c_{\lambda}-c_{\lambda+\kappa})Bk_{2,\lambda+\kappa}-A(\beta_1k_{2,\lambda}+\beta_2k_{4,\lambda+\kappa})\right)e^{\lambda y}\quad (\text{since } y<0)\cr
\geq& 0,
\end{align*}
where we have used the fact that $c_{\lambda}>c_{\lambda+\kappa}$ (see Remark \ref{remark1} $(ii)$). The lemma is thus proved.

\end{proof}

 \section{Existence of traveling wave solutions}

 In this section we suppose that $\mathcal{R}_0>1$ and present the proof of the existence of traveling wave solutions of \eqref{pde1} when $c\ge c^*$.  Right now,  the relationship between the two hypotheses $\mathcal{R}_0>1$ and $\alpha_{\max}(0)>0$ is not yet clear. The next lemma shows that in fact, $\alpha_{\max}(0)>0$ if and only if $\mathcal{R}_0>1$.

\begin{lem}\label{lemma2}
The following holds.
 $$\mathcal{R}_0=\frac{\beta_1b_1}{\mu(\phi+\mu)}+\frac{\beta\beta_2b_1b_2}{\eta^2\mu}>1\Longleftrightarrow \alpha_{\max}(0)>0.$$
\end{lem}
\begin{proof}
We first note that
\begin{equation}\label{s-eq-s1}
\mathcal{R}_0>1\Longleftrightarrow l_1>\eta l_0
\end{equation} and
\begin{align}\label{s-esq-s1}
\alpha_{max}(0)>0\,\, \Longleftrightarrow\,\,  &  4l_1>[(\eta+l_0)_+]^2-(\eta-l_0)^2=\begin{cases}
 -(\eta-l_0)^2 & \text{if}\ \eta+l_0\leq 0,\cr
 4\eta l_0  & \text{if}\ \eta+l_0\geq 0.
 \end{cases}
\end{align}

\noindent If $ \eta+l_0<0$,  it always holds that $\alpha_{\max}(0)>0$ and $ \mathcal{R}_0>1$, since $l_1>0>\max\{-(\eta-l_0)^2,\eta l_0\}$.

\noindent If $ \eta+l_0\geq 0$, by \eqref{s-eq-s1} and \eqref{s-esq-s1}, it is clear that $\alpha_{\max}(0)>0$ if and only if $\mathcal{R}_0>0$.  The lemma is thus proved.

\end{proof}

\subsection{Existence of traveling wave solutions with speed $c>c^*$.}

Throughout this subsection we will always suppose that $\mathcal{R}_0>1$, that is, $\alpha_{\max}(0)>0$ (see Lemma \ref{lemma2}). Let $0<\lambda=\lambda_{min}(c)<\lambda^*$ be given by Lemma \ref{lemma1} (i). Hence $c=c_{\lambda}$.  In the following we fix $\tilde{\lambda},\kappa,A$ and $B$ satisfying \eqref{s-eq-s2} and \eqref{s-eq-s5}. We suppose that $\underline{\bf x}_{\lambda}(y)$ and $\overline{\bf x}_{\lambda}(y)$ are given by \eqref{supper-sol-eq} and \eqref{sub-sol-eq}, respectively.  Next, let 
$$Y_0=\max\ \left\{ \frac{\ln\left(\frac{Bk_{2,\lambda+\kappa}}{k_{2,\lambda}}\right)}{\kappa},\frac{\ln\left(\frac{Bk_{4,\lambda+\kappa}}{k_{4,\lambda}}\right)}{\kappa}\  \right\}.$$
 For $x=(x_1,x_2,x_3,x_4)^T\in\R^4$ and  $\tilde x=(\tilde x_1,\tilde x_2,\tilde x_3,\tilde x_4)^T\in\R^4$ we define the order
 $$
 x\leq\tilde{x} \Longleftrightarrow x_i\leq \tilde{x}_i\ \ \forall i=1,\cdots,4.
 $$

For every $Y\ge Y_0$, let $\Omega_Y$ denotes the open interval $\Omega_Y=(-Y,Y) $,
\begin{equation}\label{eq-ey-space}
\mathcal{E}_{\lambda,Y}=\{{\bf \varphi}\in C(\overline{\Omega}_Y,\R^4)\, :\, \underline{\bf x}_{\lambda}(y)\leq \varphi(y)\leq \overline{\bf x}_{\lambda}(y),\ \forall y\in\Omega_Y,\  \varphi(\pm Y)=\underline{\bf x}_{\lambda}(\pm Y)\}.
\end{equation}
It is clear that $\mathcal{E}_{\lambda,Y}$ is a convex, closed subset of $C(\overline{\Omega}_Y,\R^4)$. Our first aim is to construct a certain self mapping function on $\mathcal{E}_{\lambda,Y}$ which is continuous and compact. Hence, we then deduce by the Schauder's fixed point theorem that any such function has a fixed point.

For every $\varphi\in\mathcal{E}_{\lambda,Y}$ we first associate the vector valued function
$$
f(y;\varphi):=\left(\begin{array}{c}
-b_1-\phi\varphi_2(y)\cr
-(\beta_1\varphi_2(y)+\beta_2\varphi_4(y))\varphi_1(y)\cr
-b_2\cr
-\beta\varphi_2(y)\varphi_3(y)
\end{array}\right)
$$
and the linear elliptic operator $\mathcal{L}[{\bf x}]=(\mathcal{L}_1[x_1],\mathcal{L}_2[x_2],\mathcal{L}_3[x_3],\mathcal{L}_4[x_4])^T$, where
$$
\mathcal{L}_1[x_1](y)= D_h\Delta_y x_1 +c_\lambda\partial_y x_1-(\mu+\beta_2\varphi_4+\beta_1\varphi_2) x_1,
$$
$$
\mathcal{L}_2[x_2](y)= D_h\Delta_y x_2 +c_\lambda\partial_y x_1-(\mu+\phi) x_2,
$$
$$
\mathcal{L}_3[x_3](y)=D_v\Delta_y x_3 +c_\lambda\partial_y x_3-(\eta+\beta\varphi_2) x_3,
$$
and
$$
\mathcal{L}_4[x_4](y)=D_v\Delta_y x_4 +c_\lambda\partial_y x_4-\eta x_4.
$$

Next, we let ${\bf x}(\cdot;\varphi,Y)$ denotes the solution of the elliptic boundary value problems

\begin{equation}\label{s-eq24}
\begin{cases}\mathcal{L}[{\bf x}(\cdot;\varphi)]=f(\cdot;\varphi),\quad y\in\Omega_Y,\cr
{\bf x}(\pm Y;\varphi)=\underline{\bf x}_{\lambda}(\pm Y).
\end{cases}
\end{equation}
By \cite[Corollary 9.18, p.243]{Trudinger}, for every $Y>Y_0$, there is a unique solution ${\bf x}(\cdot;\varphi,Y)$ of \eqref{s-eq24}.  Moreover,
\begin{equation}\label{s-eq25}
x_{i}(\cdot;\varphi,Y)\in W^{2,p}(\Omega_Y)\cap C(\overline{\Omega}_Y)\quad \forall p>1, \ i=1,\cdots,4.
\end{equation}
Furthermore, since $W^{2,p}(\Omega_Y)$ is continuously embedded in $C^{1+\alpha}(\overline{\Omega}_Y)$ ( $0<\alpha<1-\frac{1}{p}$), we conclude that $x_{i}(\cdot;\varphi,Y)\in C^{1+\alpha}(\overline{\Omega}_Y)$ for each $i=1,\cdots,4$. Hence ${\bf x}(\cdot;\varphi,Y)\in C(\overline{\Omega}_Y,\R^4)$.

The following lemmas will be needed to show that ${\bf x}(\cdot;\varphi,Y)\in \mathcal{E}_{\lambda,Y}$.

\medskip

\begin{lem}\label{lemma6} Suppose that
 $\phi\leq \frac{\beta_1b_1}{\mu}$.
 The following holds.
 \begin{itemize}
 \item[(i)]  $\underline{x}_{1,\lambda}(y)\leq x_1(y;\varphi,Y)\leq \overline{x}_{1,\lambda}(y),\quad \forall y\in\Omega_Y,\ Y>Y_0.$
 \item[(ii)]$
\underline{x}_{3,\lambda}(y)\leq x_3(y;\varphi,Y)\leq \overline{x}_{3,\lambda}(y),\quad \forall y\in\Omega_Y\ Y>Y_0.$
 \end{itemize}
 \end{lem}
 \begin{proof} $(i)$
 Since $f_1(y;\varphi)\leq 0$, we have that
 $$
\mathcal{L}_{1}[0](\cdot)=0\geq f_{1}(\cdot;\varphi).
 $$
 Thus, since $\underline{x}_{1,\lambda}(\pm Y)\geq 0$ and  $-(\mu+\beta_2\varphi_4+\beta_1\varphi_1)\leq 0$,  we conclude from the maximum principle for elliptic operators that $ x_{1}(y;\varphi,Y)\geq 0$. Next, let $Y_1\in\Omega_{Y}$ such that $\underline{x}_{1,\lambda}(y)=0$ for $Y_1< y\leq Y$ and $\underline{x}_{1,\lambda}(y)>0$ for $-Y\leq y <Y_1$. Hence the restriction of   $\underline{x}_{1,\lambda}(y)$ on the open set $(-Y,Y_1)$ is a smooth positive function and by \eqref{s-eq21} satisfies
 \begin{align*}
\mathcal{L}_{1}[\underline{x}_{1,\lambda}](y) \geq& -b_1 +\beta_2(\overline{x}_{4,\lambda}-\varphi_4)\underline{x}_{1,\lambda}(y)+\beta_1(\overline{x}_{2,\lambda}-\varphi_2)\underline{x}_{1,\lambda}(y)\cr
\geq& f_1(y;\varphi),\quad -Y<y<Y_1.
\end{align*}
 Hence since $\underline{x}_{1,\lambda}(Y_1)=0\leq x_{1}(Y_1;\varphi,Y)$ and $\underline{x}_{1,\lambda}(-Y)=x_{1}(-Y;\varphi,Y)$ and  $-(\mu+\beta_2\varphi_4+\beta_1\varphi_1)\leq 0$, by the maximum principle for elliptic operators, we conclude that $\underline{x}_{1,\lambda}(y)\leq x_{1}(y;\varphi,Y)$ for every $y\in(-Y,Y_1)$. This complete the proof of the first inequality of  $(i)$.

  Next, since $\phi\leq \frac{\beta_1b_1}{\mu}$,
  then
  $$
 \mathcal{L}_1[\overline{x}_{1,\lambda}](y)= -b_1 -\frac{\beta_2 b_1}{\mu}\varphi_4(y)-\frac{\beta_1b_1}{\mu}\varphi_2(y) \leq -b_1 -\frac{\beta_1b_1}{\mu}\varphi_2(y)\leq -b_1 -\phi\varphi_2(y)=f_1(y,\varphi).
 $$

 Hence since $ x_1(\pm Y;\varphi,Y)\leq \overline{x}_{1,\lambda}(\pm Y)$ and $-(\mu+\beta_2\varphi_4+\beta_1\varphi_1)\leq 0$, we conclude that the second inequality of $(i)$ holds by the maximum principle for elliptic equations.

$(ii)$ Observe that
 $$
\mathcal{L}_3[\overline{x}_{3,\lambda}](y) =-(\eta+\beta\varphi_2)\frac{b_2}{\eta}\le-b_2=f_2(y;\varphi)
 $$
and by \eqref{s-eq22} it holds that
$$
\mathcal{L}_3[\underline{x}_{3,\lambda}](y)\geq -b_2=f_2(y;\varphi).
$$
 Hence, since $\underline{x}_{3,\lambda}(\pm Y)\leq x_3(\pm Y;\varphi,Y)\leq \overline{x}_{2}(\pm Y)$ and $-(\eta+\beta\varphi_2)\leq 0$, following similar arguments as in case $(i)$, we conclude that $(ii)$ holds by the maximum principle for elliptic equations.
 \end{proof}

\begin{lem}\label{lemma8} Suppose that
 $\phi\leq \frac{\beta_1b_1}{\mu}$. The following holds.
\begin{itemize}
 \item[(i)]  $\underline{x}_{2,\lambda}(y)\leq x_2(y;\varphi,Y)\leq \overline{x}_{2,\lambda}(y),\quad \forall y\in\Omega_Y,\ Y>Y_0.$
 \item[(ii)]$
\underline{x}_{4,\lambda}(y)\leq x_4(y;\varphi,Y)\leq \overline{x}_{4,\lambda}(y),\quad \forall y\in\Omega_Y,\ Y>Y_0.$
\end{itemize}

\end{lem}
 \begin{proof}
 $(i)$ By Lemma \ref{lemma3} and Lemma \ref{lemma6} (i)
 we have
 \begin{align*}
 \mathcal{L}_2[\overline{x}_{2,\lambda}](y)=& -\overline{x}_{1,\lambda}(\beta_2\overline{x}_{4,\lambda}+\beta_1\overline{x}_{2,\lambda})\cr
 \leq &-(\beta_2\varphi_4+\beta_1\varphi_2)x_{1}(y;\varphi,Y),\quad y\in\Omega_Y.
 \end{align*}
 On the other hand, Lemma \ref{lemma5} and Lemma \ref{lemma6}(i) yield
 \begin{align*}
 \mathcal{L}_2[\underline{x}_{2,\lambda}](y)\ge& -\underline{x}_{1,\lambda}(\beta_2\underline{x}_{4,\lambda}+\beta_1\underline{x}_{2,\lambda})\cr
 \geq &(\beta_2\varphi_4+\beta_1\varphi_2)x_{1}(y;\varphi,Y),\quad y\in\Omega_Y.
 \end{align*}
  Hence since $\underline{x}_{2,\lambda}(\pm Y)\leq x_2(\pm Y;\varphi,Y)\leq \overline{x}_{2}(\pm Y)$ and $-(\mu+\phi)< 0$, we conclude that $(i)$ holds by the maximum principle for elliptic equations.

Part $(ii)$ follows from similar arguments as in $(i)$.
 \end{proof}

Suppose that
 $\phi\leq \frac{\beta_1b_1}{\mu}$. By Lemmas \ref{lemma6} and \ref{lemma8}, we have that the function  $ {\bf x}(\cdot;Y)\ :\  \mathcal{E}_{\lambda,Y}\ni \varphi\mapsto {\bf x}(\cdot;\varphi,Y)\in \mathcal{E}_{\lambda,Y}$ is a self-mapping.

 \begin{tm}\label{tm-02}
 Suppose that
 $\phi\leq \frac{\beta_1b_1}{\mu}$.
 Then the function    $$ {\bf x }(\cdot;Y)\ :\  \mathcal{E}_{\lambda,Y}\ni \varphi\mapsto {\bf x}(\cdot;\varphi,Y)\in \mathcal{E}_{\lambda,Y}$$
 has a fixed point.

 \end{tm}
\begin{proof} We will prove that the above function is continuous and compact.

{\bf Step 1.} Continuity. Let $\varphi,\tilde{\varphi}\in\mathcal{E}_{\lambda,Y}$ be fixed and set $w=\varphi-\tilde{\varphi}$. Let
$$
x_i=x_i(\cdot,\varphi)-x_i(\cdot,\tilde{\varphi}),\quad i=1,\cdots, 4.
$$ We have that
\begin{align*}
\begin{cases}D_h\Delta_y x_1-c_\lambda\partial_y x_1-(\mu+\beta_2\varphi_4+\beta_1\varphi_2)x_1 =(\beta_1x_1(\cdot;\tilde\varphi)-\phi)(\varphi_2-\tilde{\varphi}_2)-\beta_2(\varphi_4-\tilde{\varphi}_4)x_1(\cdot,\tilde{\varphi}),\cr
x_1(\pm Y)=0.
\end{cases}
\end{align*}
By \cite[Theorem 6.2, p.90]{Evans} there is a constant $K_{Y}(\varphi)$ such that
\begin{align*}
\|x_1\|_{W^{2,2}(-Y,Y)}\leq & K_{Y}(\varphi)\|(\phi+\beta_1x_1(\cdot;\tilde\varphi))(\varphi_2-\tilde{\varphi}_2)+\beta_2(\varphi_4-\tilde{\varphi}_4)x_1(\cdot;\tilde{\varphi})\|_{L^2(\Omega_Y)}\cr
\leq & \sqrt{2Y}K_Y(\varphi)\left(\phi+\frac{(\beta_1+\beta_2)b_1}{\mu}\right)\left(\|\varphi_2-\tilde{\varphi}_2\|_{C(\overline{\Omega}_Y)}+\|\varphi_4-\tilde{\varphi}_4\|_{C(\overline{\Omega}_Y)}\right)\cr
\leq & \sqrt{2Y}K_{Y}(\varphi)\|\varphi-\tilde{\varphi}\|_{\mathcal{E}_{\lambda,Y}}.
\end{align*}
Note that we have used the fact that $0\leq {x}_{1}(y;\tilde{\varphi})\leq \frac{b_1}{\mu} $. Similarly, it can be shown that there is a constant $K_{Y}(\varphi)\gg 1$ such that
\begin{equation*}
\|x_i\|_{W^{2,2}(-Y,Y)}\leq  \sqrt{2Y}K_{Y}(\varphi)\|\varphi-\tilde{\varphi}\|_{\mathcal{E}_{\lambda,Y}},\quad \forall i=1,\cdots,4,
\end{equation*}
which combined with the fact that $W^{2,2}(-Y,Y)$ is continuously embedded in $C^{1+\alpha}$( $0\leq \alpha<\frac{1}{2}$), yields that there is $K_{Y}(\varphi)\gg 1$ such that
\begin{equation}\label{s-eq26}
\|{\bf x}(\cdot;\varphi,Y)-{\bf x}(\cdot;\tilde{\varphi},Y)\|_{C^{1+\alpha}([-Y,Y])}\leq  \sqrt{2Y}K_{Y}(\varphi)\|\varphi-\tilde{\varphi}\|_{\mathcal{E}_{\lambda,Y}}.
\end{equation}
Hence the function $ {\bf x}(\cdot;Y)\ :\  \mathcal{E}_{\lambda,Y}\ni \varphi\mapsto {\bf x}(\cdot;\varphi,Y)\in \mathcal{E}_{\lambda,Y}$ is locally Lipschitz, hence continuous.

{\bf Step 2.} Compactness. It follows from \eqref{s-eq26} that the function $ {\bf x}(\cdot;Y)\ :\  \mathcal{E}_{\lambda,Y}\ni \varphi\mapsto {\bf x}(\cdot;\varphi,Y)\in \mathcal{E}_{\lambda,Y}$ is  compact.

Therefore by Schauder's fixed point theorem, we conclude that the function $ {\bf x}(\cdot;Y)\ :\  \mathcal{E}_{\lambda,Y}\ni \varphi\mapsto {\bf x}(\cdot;\varphi,Y)\in \mathcal{E}_{\lambda,Y}$ has a fixed point.

\end{proof}

For every $Y>Y_0$, let  ${\bf x}^*_\lambda(\cdot;Y)\in\mathcal{E}_{\lambda,Y}$ be a fixed point of the function $ {\bf x}(\cdot;Y)\ :\  \mathcal{E}_{\lambda,Y}\ni \varphi\mapsto {\bf x }(\cdot;\varphi,Y)\in \mathcal{E}_{\lambda,Y}$ given by Theorem \ref{tm-02}.  That is,
 \begin{equation}\label{s-eq27}
 \begin{cases}
0=D_h\Delta_y x_{1,\lambda}^*-c_\lambda\partial_y x^*_{1,\lambda}+b_1-(\mu+\beta_2x_{4,\lambda}^*+\beta_1x_{2,\lambda}^*)x_{1,\lambda}^*+\phi x_{2,\lambda}^*, & y\in\Omega_Y,\cr
 0=D_h\Delta_y x_{2,\lambda}^*-c_\lambda\partial_y x^*_{2,\lambda}+(\beta_1x_{1,\lambda}^*-(\phi+\mu))x_{2,\lambda}^*+\beta_2x_{1,\lambda}^*x_{4,\lambda}^*, & y\in\Omega_Y,\cr
 0=D_v\Delta_y x_{3,\lambda}^*-c_\lambda\partial_y x^*_{3,\lambda}+b_2-(\eta+\beta x_{2,\lambda}^* )x_{3,\lambda}^*, & y\in\Omega_Y, \cr
 0=D_v\Delta_y x_{4,\lambda}^*-c_\lambda\partial_y x^*_{4,\lambda}-\eta x_{4,\lambda}^*+\beta x_{2,\lambda}^*x_{3,\lambda}^*, & y\in\Omega_Y,\cr
 {\bf x}_{\lambda}^*(\pm Y;Y)=\underline{\bf x}_{\lambda}(\pm Y).
 \end{cases}
 \end{equation}

\begin{tm}\label{tm-03} Suppose that
 $\phi\leq \frac{\beta_1b_1}{\mu}$. Let $ {\bf x}_{\lambda}^*(\cdot;Y)$ be given by \eqref{s-eq27}. For every $Y>Y_0 $ there is a constant $K_Y$ such that \begin{equation}\label{s-eq30}
\|{\bf x}_\lambda^*(\cdot;\tilde{Y})\|_{C^{2+\alpha}(\overline{\Omega}_Y)}\leq  K_{Y},\quad \forall \tilde{Y}\geq 2Y.
\end{equation}

 \end{tm}
 \begin{proof}
 Recall from the above that
 \begin{equation} \label{s-eq31}
 \underline{\bf x}_{\lambda}(y)\leq {\bf x}_\lambda^*(y;\tilde{Y})\leq \overline{\bf x}_{\lambda}(y)\quad \forall \ y\in\Omega_Y,\,\,  Y<\tilde{Y},
\end{equation}
 for any  given  $\tilde{Y}>Y>Y_0$.
Therefore for any given  $Y>Y_0$ there is a constant $K_{1,Y}$ depending only of $Y$  such that
\begin{equation}\label{s-eq28}
\|{\bf x}_\lambda^*(\cdot;\tilde{Y})\|_{C(\overline{\Omega}_Y)}\leq  K_{1,Y},\quad \forall\ \tilde{Y}>Y.
\end{equation}
\medskip

Hence by \cite[Theorem 9.11, page 235]{Trudinger}, for every $Y>Y_0$, there is a constant $K_{2,Y}\gg K_{1,Y}$ such that
\begin{equation*}
\|{\bf x}_\lambda^*(\cdot;\tilde{Y})\|_{W^{2,2}(\Omega_Y)}\leq  K_{2,Y},\quad \forall \tilde{Y}\geq 2Y,
\end{equation*}
which together with the fact that $W^{2,2}(\Omega_Y)$ is continuously embedded in $C^{1+\alpha}$( $0\leq \alpha<\frac{1}{2}$), yields that there is $K_{3,Y}\gg K_{2,Y}$ such that
\begin{equation}\label{s-eq29}
\|{\bf x}_\lambda^*(\cdot;\tilde{Y})\|_{C^{1+\alpha}(\overline{\Omega}_Y)}\leq  K_{3,Y},\quad \forall \tilde{Y}\geq 2Y.
\end{equation}
Therefore, by \eqref{s-eq28}, \eqref{s-eq29}, and \eqref{s-eq27}, we conclude that there is $K_{Y}\gg K_{4,Y}$ such that
\begin{equation*}
\|{\bf x}_\lambda^*(\cdot;\tilde{Y})\|_{C^{2+\alpha}(\overline{\Omega}_Y))}\leq  K_{Y},\quad \forall \tilde{Y}\geq 2Y.
\end{equation*}
So, \eqref{s-eq30} holds.

\end{proof}

The following lemmas will be used to complete the proof of Theorem \ref{main-tm-2}(ii).

\begin{lem}\label{lemma10}
Suppose that $\mathcal{R}_0>1$ and let $E_1=(x_1^{**},x_2^{**},x_3^{**},x_4^{**})^T$ denote the endemic equilibrium.  Define
\begin{align}
\mathcal{G}({\bf x})=\sum_{i=1}^4a_i\left(1-\frac{x_i^{**}}{x_{i}}\right)F_i({\bf x})
\end{align}
for every ${\bf x}\in [\R^4]^+$  satisfying $x_1+x_2=\frac{b_1}{\mu}$ and $x_3+x_4=\frac{b_2}{\eta}$, where $a_1=a_2=\beta x_1^{**}x_4^{**}$ and $a_3=a_4=\beta_2x_3^{**}x_2^{**}+\beta_1x_1^{**}x_4^{**}$ and $F({\bf x})$ is defined by \eqref{F-def}. Then

\begin{align*}
&\mathcal{G}({\bf x}) +a_1(\mu+\phi)\frac{(x_1-x_1^{**})^2}{x_1}+a_3\eta\frac{(x_3-x_3^{**})^2}{x_3}\cr
=&-a_1a_3\left[\frac{x_1^{**}}{x_1}
+\frac{x_2^{**}}{x_2}+\frac{x_1x_{2}^{**}(\beta_2x_4+\beta x_2)}{x_1^{**}x_2(\beta_2 x_4^{**}+\beta x_2^{**})} +\frac{x_3x_{2}(\beta_2x_4^{**}+\beta x_2^{**})}{x_3^{**}x_2^{**}(\beta_2 x_4+\beta x_2)} -4\right]\cr
\leq & 0 ,
\end{align*}
and $\mathcal{G}({\bf x})=0$ if and only if $ {\bf x}=E_1$.
\end{lem}
\begin{proof}
The lemma follows by proper modification of the arguments used to prove \cite[Theorem 3.7,page 7]{CaiLi}.
\end{proof}

\begin{lem}\label{lemma11}\cite[Lemma 2.2]{SaSh} Let $d,\kappa>0$, ${c}\in\R$. For every $u\in C^b_{\rm unif}(\R)$ with $u\ge 0$, let $v(y)\in c^{2,b}_{\rm uinf}(\R)$ denotes the solution of
$$
0=d v'' +{c}v'-\kappa v+u.
$$
Then
$$
|v'(y)|\leq \frac{\left(\sqrt{c^2+4\kappa}+|c|\right)}{2d}v(y),\quad \forall\ y\in\R.
$$
Therefore, the following Harnack's inequality holds
$$
v(y)\leq v(\tilde{y})e^{|y_1-y_2|\frac{\sqrt{c^2+4\kappa}+|c|}{2d}},\quad \forall y,\tilde{y}\in[y_1,y_2].
$$

\end{lem}

\medskip

\begin{proof}[Proof of Theorem \ref{main-tm-02}(i) for $c>c^*$] Consider the sequence of functions $\{{\bf x}_\lambda^*(\cdot;m)\}_{m\ge Y_0}$. By Theorem \ref{tm-03} and the Arzela-Ascoli Theorem, there is a subsequence $ \{{\bf x}_\lambda^*(\cdot;m')\}_{m'\ge Y_0}$ of the sequence $\{{\bf x}_\lambda^*(\cdot;m)\}_{m\ge Y_0}$ and a function $\tilde{\bf x}_{\lambda}^{*}\in C^{2}(\R)$ such that $ {\bf x}_\lambda^*(\cdot;m')\to \tilde{\bf x}_{\lambda}^{*}$ locally uniformly in $C^2(\R)$. Moreover, the function $\tilde{\bf x}_{\lambda}^{*}$ satisfies  \eqref{s-eq27} in $\R$.
 Hence ${\bf x}_\lambda(t,y)={\bf x}_\lambda^{**}(y+c_\lambda t) $ is a traveling wave solution of \eqref{pde1}. Next we show that ${\bf x}_{\lambda}^{**}(y)$ connects $E_0$ and $E_1$.
Recall from \eqref{s-eq31} that
$$
\underline{\bf x}_{\lambda}(y)\leq {\bf x}_\lambda^{*}(y;m')\leq \overline{\bf x}_{\lambda}(y)\quad \forall y\in\Omega_{m'}.
$$
Hence, letting $m\to\infty$ yields
$$
\underline{\bf x}_{\lambda}(y)\leq \tilde{\bf x}_{\lambda}^{*}(y)\leq \overline{\bf x}_{\lambda}(y)\quad \forall y\in\R.
$$
And observe that
$$
\lim_{y\to-\infty}\underline{\bf x}_{\lambda}(y)=\lim_{y\to-\infty}\overline{\bf x}_{\lambda}(y)=E_0,
$$
whence
$$
\lim_{y\to-\infty}\tilde{\bf x}_{\lambda}^{*}(y)=\lim_{y\to-\infty}\overline{\bf x}_{\lambda}(y)=E_0.
$$
Therefore $ \tilde{\bf x}_{\lambda}^*(t,y)=\tilde{\bf x}_{\lambda}^{*}(y+c_{\lambda}t) $ is a traveling wave solution of \eqref{pde1} with speed $c_{\lambda}$ connecting $E_0$ at one end.

Next we discussed the behavior of $ \tilde{\bf x}_{\lambda}^*(y) $ as $y\to\infty$ in three steps.

{\bf Step 1.} We claim that
\begin{equation}\label{ss-ss-1}
\tilde{x}^{*}_{1,\lambda}(y)+\tilde x_{2,\lambda}^{*}(y)= \frac{b_1}{\mu} \quad \text{and}\quad  \tilde x_{3,\lambda}^{*}(y)+\tilde x_{4,\lambda}^{*}(y)= \frac{b_2}{\eta}, \quad \forall\ y\in\R.
\end{equation}
Indeed, observe that
\begin{align*}
\mu\tilde x_{1,\lambda}^{*}-c_\lambda\partial_y\tilde x^{*}_{1,\lambda} -D_h\Delta_y\tilde x^{*}_{1,\lambda}=b_1-(\beta_2\tilde x^{*}_{4,\lambda}+\beta_1\tilde x^{*}_{2,\lambda})\tilde x_{1,\lambda}^{*}+\phi \tilde x_{2,\lambda}^{*}
\end{align*}
and
$$
\mu \tilde x_{2,\lambda}^{*}-c_\lambda\partial_y \tilde x^{*}_{2,\lambda}-D_h\Delta_y \tilde x_{2,\lambda}^{*}=(\beta_2\tilde x^{*}_{4,\lambda}+\beta_1\tilde x^{*}_{2,\lambda})\tilde x_{1,\lambda}^{*} -\phi\tilde x_{2,\lambda}^{*}.
$$
By adding up these two equations side by side yield
$$
\mu (\tilde x_{1,\lambda}^{*}+\tilde x_{2,\lambda}^{*})-c_\lambda\partial_y (\tilde x_{1,\lambda}^{*}+\tilde x_{2,\lambda}^{*})-D_h\Delta_y (\tilde x_{1,\lambda}^{*}+\tilde x_{2,\lambda}^{*})=b_1.
$$
Let $\{e^{t\Delta}\}_{t\geq 0}$ denote the analytic semigroup generated by the Laplace operator $\Delta$, on $C^b_{\rm unif}(\R)$. Since by construction, there is a positive constant $K$ so that  $0<(\tilde x_{1,\lambda}^{*}+\tilde x_{2,\lambda}^{*})(y)<Ke^{\lambda y}$ for every $y\in\R$, then it follows that (see \cite[Chapter 1]{Freidman})
\begin{align*}\label{s-eq-ss-2}
(\tilde x_{1,\lambda}^{*}+\tilde x_{2,\lambda}^{*})(y)=&\int_{0}^{\infty}e^{-\mu t}[e^{D_ht\Delta}(b_1)]dt=\frac{b_1}{\mu},\quad \forall\ y\in\R.
%=&\frac{b_1}{\mu}+\int_{0}^{\infty}e^{-\mu t}T(d_1t)[\phi x^{**}_{2,\lambda}-(\beta_2x^{**}_{4,\lambda}+\beta_1x^{**}_{2,\lambda})x_{1,\lambda}^{**}](\cdot+c_{\lambda}t)dt
\end{align*}
Similarly, observe that
$$
\eta (\tilde x_{3,\lambda}^{*}+\tilde x_{4,\lambda}^{*})-c_\lambda\partial_y (\tilde x_{3,\lambda}^{*}+\tilde x_{4,\lambda}^{*})-D_v\Delta_y (\tilde x_{3,\lambda}^{*}+\tilde x_{4,\lambda}^{*})=b_2.
$$
Hence, the second equation of Step 1 also holds.

{\bf Step 2.} It holds that
\begin{equation}\label{ss-ss-001}
\tilde x_{1,\lambda}^{*}(y)\geq  \frac{b_1}{\mu +\frac{\beta_2 b_2}{\eta} +\frac{\beta_1 b_1}{\mu}} \quad \text{and}\quad \tilde x_{3,\lambda}^{*}(y)\geq \frac{b_2}{\eta+\frac{\beta b_2}{\eta}}, \quad \forall\ y\in\R.
\end{equation}
Indeed, observe that
\begin{align*}
\left(\mu +\frac{\beta_2 b_2}{\eta} +\frac{\beta_1 b_1}{\mu}\right)\tilde x_{1,\lambda}^{*}-c_\lambda\partial_y\tilde x^{*}_{1,\lambda} -D_h\Delta_y\tilde x^{*}_{1,\lambda}=b_1 +\beta_2\left(\frac{b_2}{\eta}-\tilde x^{*}_{4,\lambda}\right)+\beta_1\left(\frac{b_1}{\mu}-\tilde x^{*}_{2,\lambda}\right)+\phi \tilde x_{2,\lambda}^{*}.
\end{align*}
Hence, it follows from step 1 and positivity of $\{e^{t\Delta}\}_{t\ge 0}$ that
\begin{align*}
\tilde x^{*}_{1,\lambda}(y)=&\int_{0}^{\infty}e^{-\left(\mu +\frac{\beta_2 b_2}{\eta} +\frac{\beta_1 b_1}{\mu}\right)t}\left[e^{D_ht\Delta}\left( b_1 +\beta_2\left(\frac{b_2}{\eta}-\tilde x^{*}_{4,\lambda}\right)+\beta_1\left(\frac{b_1}{\mu}-\tilde x^{*}_{2,\lambda}\right)+\phi \tilde x_{2,\lambda}^{*}\right)(y+c_\lambda t)\right]dt\cr
\geq & \int_{0}^{\infty}e^{-\left(\mu +\frac{\beta_2 b_2}{\eta} +\frac{\beta_1 b_1}{\mu}\right)t}[e^{D_ht\Delta} b_1]dt=\frac{b_1}{\mu +\frac{\beta_2 b_2}{\eta} +\frac{\beta_1 b_1}{\mu}}.
\end{align*}
Similarly, we have
$$
\left(\eta+\frac{\beta b_2}{\eta}\right)\tilde x_{3,\lambda}^{*}-c_\lambda\partial_y\tilde x^{*}_{3,\lambda}- D_v\Delta_y\tilde x_{3,\lambda}^{*}=b_2+\beta\left(\frac{b_2}{\eta}- \tilde x_{2,\lambda}^{*} \right)\tilde x_{3,\lambda}^{*}.
$$
Hence, similar arguments as in the previous case also yield the second inequality of Step 2.

\vspace{1 cm}
{\bf Step 3.} We complete the proof of the theorem in this step. Recall that $\tilde{\bf x}^*_{\lambda}$ satisfies
\begin{equation}\label{z-e-z1}
\begin{cases}
 0=D_h\frac{d^2}{dy^2}\tilde x_{2,\lambda}^*-c_\lambda\frac{d}{dy}\tilde x^*_{2,\lambda}-(\phi+\mu)\tilde x_{2,\lambda}^*+(\beta_1\tilde x_{2,\lambda}^*+\beta_2\tilde x_{4,\lambda}^*)\tilde x_{1,\lambda}^*, & y\in\R,\cr
 0=D_v\frac{d^2}{•dy^2}\tilde x_{4,\lambda}^*-c_\lambda\frac{d}{•dy}\tilde x^*_{4,\lambda}-\eta \tilde x_{4,\lambda}^*+\beta\tilde x_{2,\lambda}^*\tilde x_{3,\lambda}^*, & y\in\R.
\end{cases}
\end{equation}

Hence by Lemma \ref{lemma11}, it holds that
\begin{equation}\label{z-e-z2}
\frac{|\frac{d}{dy}\tilde{x}^*_{2,\lambda}(y) |}{|\tilde{x}^*_{2,\lambda}(y)|}\le \frac{\left( \sqrt{4(\mu+\phi)+c_\lambda^2}+c_{\lambda}\right)}{2D_h}\quad \text{and}\quad \frac{|\frac{d}{dy}\tilde{x}^*_{4,\lambda}(y) |}{|\tilde{x}^*_{4,\lambda}(y)|}\leq \frac{\left(\sqrt{4\eta+c_\lambda^2}+c_\lambda\right)}{2{D}_v},
\end{equation}
which together with Step 1 and Step 2 yield that
\begin{equation}\label{z-e-z003}
\frac{|\frac{d}{dy}\tilde{x}^*_{1,\lambda}(y) |}{|\tilde{x}^*_{1,\lambda}(y)|}=\frac{|\frac{d}{dy}\tilde{x}^*_{2,\lambda}(y) |}{|\tilde{x}^*_{1,\lambda}(y)|}\le \frac{\left( \sqrt{4(\mu+\phi)+c_\lambda^2}+c_{\lambda}\right)\left(\mu +\frac{\beta_2b_2}{\eta}+\frac{\beta_1b_1}{\mu}\right)}{2\mu D_h}
\end{equation}
and
\begin{equation}\label{z-e-z004}
\frac{|\frac{d}{dy}\tilde{x}^*_{3,\lambda}(y) |}{|\tilde{x}^*_{3,\lambda}(y)|}=\frac{|\frac{d}{dy}\tilde{x}^*_{4,\lambda}(y) |}{|\tilde{x}^*_{3,\lambda}(y)|}\le \frac{\left( \sqrt{4\eta+c_\lambda^2}+c_{\lambda}\right)\left(\eta +\frac{\beta b_2}{\eta}\right)}{2\eta {D}_v}.
\end{equation}

Next,  define the Lyaponov function
 \begin{equation}\label{z0-e-z3}
\mathcal{V}(y)= \sum_{i=1}^4 a_i\left( d_i\tilde{x}^{*'}_{i,\lambda}\left(\frac{x_i^{**}}{\tilde{x}^*_{i,\lambda}(y)}-1\right) +c_\lambda x_{i}^{**}\mathcal{L}\left(\frac{\tilde{x}^*_{i,\lambda}(y)}{x_i^{**}}\right) \right)
 \end{equation}
where $d_1=d_2=D_h$, $d_3=d_4=D_v$, $\mathcal{L}(s)=s-1-\ln(s)$, $a_1,\cdots,a_4$ are given by Lemma \ref{lemma10} , and $E_1=(x_1^{**},x_2^{**},x_3^{**},x_4^{**})^T$ is the endemic equilibrium. It holds that
\begin{align}\label{z-e-z3}
\frac{d}{dy}\mathcal{V}(y)=& \sum_{i=1}^4 a_i\left( d_i\tilde{x}_{i,\lambda}^{*''}\left(\frac{x_i^{**}}{\tilde{x}^*_{i,\lambda}(y)}-1\right) -d_ix_i^{**}\left(\frac{\tilde{x}_{i,\lambda}^{*'}}{\tilde{x}^*_{i,\lambda}}\right)^2 +c_\lambda{\tilde{x}_{i,\lambda}^{*'}(y)}\left(1-\frac{x_i^{**}}{\tilde{x}^*_{i,\lambda}(y)}\right) \right) \cr
=& \sum_{i=1}^4 a_i\left( \left(-c_{\lambda}\tilde{x}_{i,\lambda}^{*'}+F_i({\bf \tilde{x}^*_{\lambda}}) \right)\left(1-\frac{x_i^{**}}{\tilde{x}^*_{i,\lambda}(y)}\right) -d_ix_i^{**}\left(\frac{\tilde{x}_{i,\lambda}^{*'}}{\tilde{x}^*_{i,\lambda}}\right)^2 +c_\lambda\tilde{x}^{*'}_{i,\lambda}(y)\left(1-\frac{x_i^{**}}{\tilde{x}^*_{i,\lambda}(y)}\right) \right)\cr
=&\mathcal{G}({\bf \tilde{x}^*_{\lambda}}(y))-\sum_{i=1}^4d_ix_{i}^{**}\left( \frac{\tilde{x}^{*'}_{i,\lambda}}{\tilde{x}^*_{i,\lambda}}\right)^2\leq -\left(a_1(\mu+\phi)\frac{(\tilde{x}^*_{1,\lambda}-x_1^{**})^2}{\tilde{x}^*_{1,\lambda}}+a_3\eta\frac{(\tilde{x}^*_{3,\lambda}-x_3^{**})^2}{\tilde{x}^*_{3,\lambda}} \right)\cr
\end{align}
where we have used Lemma \ref{lemma10}. Notice  by \eqref{z-e-z2},\eqref{z-e-z003}, and \eqref{z-e-z004} we have that
\begin{align}\label{z-e-z30}
\left| \sum_{i=1}^4 a_i d_i\tilde{x}^{*'}_{i,\lambda}\left(\frac{x_i^{**}}{\tilde{x}^*_{i,\lambda}(y)}-1\right)\right|\le  M_{c_\lambda}\sum_{i=1}^4d_i\left(x_{i}^{**}+\frac{b_1}{\mu}+\frac{b_2}{\eta}\right),
\end{align}
where
$$
M_{c_\lambda}=\frac{\left( \sqrt{4\eta+c_\lambda^2}+c_{\lambda}\right)\left(\eta +\frac{\beta b_2}{\eta}\right)}{\eta D_v}+\frac{\left( \sqrt{4(\mu+\phi)+c_\lambda^2}+c_{\lambda}\right)\left(\mu +\frac{\beta_2b_2}{\eta}+\frac{\beta_1b_1}{\mu}\right)}{\mu D_h}.
$$
Hence since   $\mathcal{L}(s)\geq 0$, then
$$
\mathcal{V}(y)\geq -M_{c_\lambda},\quad \forall\ y\in\R.
$$
Therefore by \eqref{z-e-z3}
\begin{align*}\label{z-e-z31}
&\int_{y_0}^y\left(a_1(\mu+\phi)\frac{(\tilde{x}^*_{1,\lambda}(s)-x_1^{**})^2}{\tilde{x}^*_{1,\lambda}(s)}+a_3\eta\frac{(\tilde{x}^*_{3,\lambda}(s)-x_3^{**})^2}{\tilde{x}^*_{3,\lambda}(s)} \right)ds\cr
\leq & \mathcal{V}(y_0)-\mathcal{L}(y)\cr
\leq & \mathcal{L}(y_0)+M_{c_\lambda}\sum_{i=1}^4d_i\left(x_{i}^{**}+\frac{b_1}{\mu}+\frac{b_2}{\eta}\right)
\end{align*}
for every $y>y_0$. Whence  for $y_0\in\R$
\begin{equation}\label{z-e-z32}
\int_{y_0}^{\infty}\left(a_1(\mu+\phi)\frac{(\tilde{x}^*_{1,\lambda}(s)-x_1^{**})^2}{\tilde{x}^*_{1,\lambda}(s)}+a_3\eta\frac{(\tilde{x}^*_{3,\lambda}(s)-x_3^{**})^2}{\tilde{x}^*_{3,\lambda}(s)} \right)ds\leq  \mathcal{V}(y_0)+M_{c_\lambda}\sum_{i=1}^4d_i\left(x_{i}^{**}+\frac{b_1}{\mu}+\frac{b_2}{\eta}\right).\\
\end{equation}
Since $ s\mapsto a_1(\mu+\phi)\frac{(\tilde{x}^*_{1,\lambda}(s)-x_1^{**})^2}{\tilde{x}^*_{1,\lambda}(s)}+a_3\eta\frac{(\tilde{x}^*_{3,\lambda}(s)-x_3^{**})^2}{\tilde{x}^*_{3,\lambda}(s)} $ belongs to $C^{1,b}_{\rm unif}(\R)$, we conclude that
$$
\lim_{s\to\infty}\left( a_1(\mu+\phi)\frac{(\tilde{x}^*_{1,\lambda}(s)-x_1^{**})^2}{\tilde{x}^*_{1,\lambda}(s)}+a_3\eta\frac{(\tilde{x}^*_{3,\lambda}(s)-x_3^{**})^2}{\tilde{x}^*_{3,\lambda}(s)} \right)=0,
$$
which together with Steps 1 and 2 yield that
$$
\lim_{y\to\infty}{\bf \tilde{\bf x}_{\lambda}^*}(y)=E_1.
$$

% Moreover,  $ \frac{d}{dy}\mathcal{V}({\bf x})(y)=0$ if and only if then ${\bf x}(y)\equiv E_1$.  Therefore, by the invariance   principle of LaSalle,  we conclude that $$ \lim_{y\to\infty}\tilde{x}^*_{\lambda}(y)=E_1.$$

\end{proof}

\subsection{Existence of traveling wave solutions with minimum wave speed $c^*$}

In this subsection we present the proof of traveling wave solutions with speed $c^*$ connecting $E_0$ and $E_1$.

We start with the following lemma.

\begin{lem}\label{lemma12}
Let ${\bf x}(t,y)={\bf x}(y+ct)$ be a traveling wave solution of \eqref{pde1} with speed $c\in\mathbb{R}$. Then there is a constant $m_c\gg 1$ such that

\begin{equation}\label{ax-eq02}
\frac{1}{m_c}x_{2}(y)\leq x_{4}(y)\leq m_c x_{2}(y),\forall y\in\R.
\end{equation}
Moreover, $m_c$ is bounded on every bounded interval.

\end{lem}

\begin{proof}[Proof of \eqref{ax-eq02}.] Recall that ${\bf x}$ satisfies
\begin{equation}\label{0zz-e-z1}
\begin{cases}
 0=D_h\frac{d^2}{dy^2}x_{2}-c\frac{d}{dy}x_{2}-(\phi+\mu)x_{2}+(\beta_1x_{2}+\beta_2 x_{4}) x_1, & y\in\R,\cr
 0=D_v\frac{d^2}{•dy^2} x_4-c\frac{d}{•dy}x_{4}-\eta x_{4}+\beta x_{2} x_{3}, & y\in\R.
\end{cases}
\end{equation}
By the Harnack's inequality of Lemma \ref{lemma11}, there is a constant $K=K(c)$ such that
\begin{equation}\label{1zz-e-z1}
x_{i}(s)\geq K x_{i}(z),\quad \forall z\in\R, \ s\in[z-1,z+1],\ i=2,4.
\end{equation}
Next, observe that
\begin{align}\label{2zz-e-z1}
x_2(z)=&\frac{D_h}{2(\sqrt{c^{2}+4(\mu+\phi)})}\int_{\R}\left((\beta_1x_{2}+\beta_2 x_{4}) x_1(s)\right)e^{-\frac{\sqrt{c^{2}+4(\phi+\mu)}}{2D_h}|z-s|-\frac{c}{2D_h}(z-s)}ds\cr
\geq& \frac{D_h}{2(\sqrt{c^{2}+4(\mu+\phi)})}\int_{B(z,1)}\left((\beta_1x_{2}+\beta_2 x_{4}) x_1(s)\right)e^{-\frac{\sqrt{c^{2}+4(\phi+\mu)}}{2D_h}|z-s|-\frac{c}{2D_h}(z-s)}ds\cr
\geq& \frac{D_hK(\beta_1x_{2}(z)+\beta_2x_4(z))b_1}{2(\sqrt{c^{2}+4(\mu+\phi)})(\mu+\frac{\beta_1b_1}{\mu}+\frac{\beta_2b_2}{\eta})}\int_{B(z,1)}e^{-\frac{\sqrt{c^{2}+4(\phi+\mu)}}{2D_h}|z-s|-\frac{c}{2D_h}(z-s)}ds\cr
=& \frac{D_hK(\beta_1x_{2}(z)+\beta_2x_4(z))b_1}{2(\sqrt{c^{2}+4(\mu+\phi)})(\mu+\frac{\beta_1b_1}{\mu}+\frac{\beta_2b_2}{\eta})}\int_{B(0,1)}e^{-\frac{\sqrt{c^{2}+4(\phi+\mu)}}{2D_h}|s|+\frac{c}{2D_h}s}ds\cr
\geq& \left[ \frac{\beta_2D_hKb_1}{2(\sqrt{c^{2}+4(\mu+\phi)})(\mu+\frac{\beta_1b_1}{\mu}+\frac{\beta_2b_2}{\eta})}\int_{B(0,1)}e^{-\frac{\sqrt{c^{2}+4(\phi+\mu)}}{2D_h}|s|+\frac{c}{2D_h}s}ds \right]x_4(z)
\end{align}
where we have used \eqref{1zz-e-z1} and \eqref{ss-ss-001}.

Similarly, note that

\begin{align}\label{3zz-e-z1}
x_4(z)=&\frac{\beta D_v}{2(\sqrt{c^{2}+4\eta})}\int_{\R}\left(x_{2}(s) x_3(s)\right)e^{-\frac{\sqrt{c^{2}+4\eta }}{2D_v}|z-s|-\frac{c}{2D_v}(z-s)}ds\cr
\geq& \frac{D_v\beta K b_2 x_2(z)}{2(\sqrt{c^{2}+4\eta})(\eta+\frac{\beta b_2}{\eta})}\int_{B(z,1)}e^{-\frac{\sqrt{c^{2}+4\eta }}{2D_v}|z-s|-\frac{c}{2D_v}(z-s)}ds\cr
=&\left[\frac{D_v\beta K b_2 }{2(\sqrt{c^{2}+4\eta})(\eta+\frac{\beta b_2}{\eta})}\int_{B(0,1)}e^{-\frac{\sqrt{c^{2}+4\eta }}{2D_v}|s|+\frac{c}{2D_v}s}ds \right]x_2(z),
\end{align}
where we have also used  \eqref{1zz-e-z1} and \eqref{ss-ss-001}. Therefore \eqref{ax-eq02} follows from both \eqref{2zz-e-z1} and \eqref{3zz-e-z1}.
\end{proof}

Next, we present the proof of Theorem \ref{main-tm-02} (ii).
\begin{proof}[Proof of Theorem  \ref{main-tm-02}(ii) for $c=c^*$]

Let $c_n>c^*$ with $c_n\to c^*$ as $n\to\infty$. For $n\geq 1$, let ${\bf x}(t,y;n)={\bf x}(y+c_nt;n)$ denote a traveling wave solution with speed $c_n$ connecting $E_1$ and $E_0$ constructed in the previous section. For each $n\geq 1$ let
$$
y_n=\min\left\{y\in\R\ :\ x_{2}(y;n)\geq \frac{x_2^{**}}{2}\right\}.
$$
Hence
\begin{equation}\label{eq-last1}
x_{2}(y_n;n)=\frac{x_2^{**}}{2} \quad \text{and}\quad x_{2}(y+y_n;n)\le \frac{x_2^{**}}{2},\quad y\leq 0.
\end{equation}
Consider the sequence of functions $\{{\bf x}(\cdot+y_n;n)\}_{n\ge 1}$. Since  by Step 1 in the proof of existence of ${\bf x(\cdot;n)}$ we have that
$$
\|x_i(\cdot;n)\|_{\infty}\leq \frac{b_1}{\mu}+\frac{b_2}{\eta},
$$
then by standard diagonalization arguments using parabolic estimates, without loss of generality, we may suppose that there is ${\bf x}(y)\in C^2_{\rm unif}(\R) $ such that ${\bf x}(y;n)\to {\bf x}(y) $ locally uniformly in $C^2(\R)$.  Moreover ${\bf x}(y)$ is a nontrivial bounded traveling  wave solution of \eqref{pde1} with speed $c^*$.  Similarly, as in the above, we have that ${\bf x}(y)$ satisfies \eqref{ss-ss-1}, \eqref{ss-ss-001}, \eqref{z-e-z2}, \eqref{z-e-z003}, and \eqref{z-e-z004}. Therefore the Lyapunov function \eqref{z0-e-z3} is well defined with ${\bf\tilde{x}}_{\lambda}^*$ replaced by ${\bf x}$. And by \eqref{z-e-z32}, we have
\begin{equation}\label{z-e-z322}
\int_{y_0}^{\infty}\left(a_1(\mu+\phi)\frac{(x_{1}(s)-x_1^{**})^2}{x_{1}(s)}+a_3\eta\frac{(x_{3}(s)-x_3^{**})^2}{x_{3}(s)} \right)ds\leq  \mathcal{V}(y_0)+M_{c^*}\sum_{i=1}^4d_i\left(x_{i}^{**}+\frac{b_1}{\mu}+\frac{b_2}{\eta}\right),\\
\end{equation}
for any $y_0\in\R$, which implies as in the above that
$$
\lim_{y\to\infty}{\bf x}(y)=E_1.
$$
But by \eqref{eq-last1}, it holds that
$$
x_2(y)\leq \frac{x_2^{**}}{2},\quad \forall y\geq 0,
$$
whence $\lim_{y\to-\infty}{\bf x}(y)\neq E_1$. Thus by \eqref{z-e-z322}, we must have that
$$
\lim_{y_0\to-\infty}\mathcal{V}(y_0)=\infty.
$$
Hence, by \eqref{z-e-z30}, \eqref{ss-ss-001}, and \eqref{ss-ss-1} we must have that
$$
\lim_{y\to\-\infty}\left( \mathcal{L}\left(\frac{x_{2}(y)}{x_2^{**}}\right)+\mathcal{L}\left(\frac{x_{4}(y)}{x_4^{**}}\right)\right)=\infty.
$$
Recalling the definition of $\mathcal{L}(s)=s-1-\ln(s)$, the fact that $\|x_2\|_{\infty}+\|x_4\|_{\infty}\leq \frac{b_1}{\mu}+\frac{b_2}{\eta},$ we deduce that
\begin{equation}\label{ax-eq01}
\lim_{y\to-\infty}x_{2}(y)x_{4}(y)=0.
\end{equation}

Now by  \eqref{ax-eq02} and \eqref{ax-eq01}, we obtain that
 $$
\lim_{y\to-\infty}x_{2}(y)=\lim_{y\to-\infty}x_4(y)=0,
 $$
 which combined with \eqref{ss-ss-1} yield $$\lim_{y\to-\infty}{\bf x}(y)=E_0. $$
 Therefore ${\bf x}(t,y)={\bf x}(y+c^*t)$ is a traveling wave solution connecting $E_0$ and $E_1$.
\end{proof}

\section{Non-existence of traveling wave solutions}
In this section, we present the proof of non-existence of traveling wave solutions.

\begin{proof}[Proof of Theorem \ref{main-tm-1}] Suppose to the contrary that  ${\bf x}(t,y)={\bf x}(y+ct)$ is a traveling wave solution of \eqref{pde1}  connecting $E_0$ at one end, say at $y=-\infty$, with some speed $c\in\R$. It follows from the proof of Theorem \eqref{main-tm-02}(i) Step 1 that   ${\bf x}(t,y)$ satisfies \eqref{ss-ss-1}. So,
$\|x_{1}\|_{\infty}\leq \frac{b_1}{\mu}$ and $\|x_{3}\|_{\infty}\leq \frac{b_2}{\eta}$. Thus ${\bf x}(t,y)$ satisfies
\begin{equation}\label{s-eq-w-2}
\begin{cases}
 \frac{\partial x_2}{\partial t}\leq D_h\Delta_y x_2+ (\frac{\beta_1b_1}{\mu}-(\phi+\mu))x_2+\frac{\beta_2b_1}{\mu} x_4,& t>0,\ y\in\R,\cr
 \frac{\partial x_4}{\partial t}\leq D_v\Delta_y x_4-\eta x_4+\frac{\beta b_2}{\eta} x_2,& t>0,\ y\in\R.
\end{cases}
\end{equation}
Next, let $(k_{2,0},k_{4,0})^T$ denote a  positive eigenvector associated with $\alpha_{\max}(0)$ satisfying $$(\|x_2\|_{\infty},\|x_4\|_{\infty})^T\leq (k_{2,0},k_{4,0})^T.$$
By \eqref{s-eq4}, we note that the space homogeneous function
$$
(\tilde{x}_2(t,y),\tilde{x}_4(t,y))^T=e^{t\alpha_{\max}(0)}(k_{2,0},k_{4,0})^T
$$
is a super solution of \eqref{s-eq-w-2} satisfying
$$
({x}_2(0,y),{x}_4(0,y))^T\leq (\tilde{x}_2(0,y),\tilde{x}_4(0,y))^T,\quad \forall\ y\in\R.
$$
Hence by comparison principle for cooperative parabolic systems, we conclude that
$$
({x}_2(y+ct),{x}_4(y+ct))^T\leq e^{t\alpha_{\max}(0)}(k_{2,0},k_{4,0})^T,\quad \forall\ y\in\R, \ t\geq 0.
$$
Hence
$$
0\leq ({x}_2(y),{x}_4(y))^T\leq e^{t\alpha_{\max}(0)}(k_{2,0},k_{4,0})^T,\quad \forall\ y\in\R, \ t\geq 0.
$$
Letting $t\to\infty$, we obtain that
$$
({x}_2(y),{x}_4(y))^T=0,\quad \forall\ y\in\R,
$$
since $\alpha_{\max}(0)<0$.
\end{proof}

\begin{proof}[Proof of Theorem \ref{main-tm-2}(i)]

Let ${\bf x}(t,y)={\bf x}(y+ct)$ be a traveling wave solution with speed $c$ connecting $E_0$ at one end with $|c|<c^*$. Then
$$
\begin{cases}
 0=D_h\frac{d^2}{dy^2}x_{2}-c\frac{d}{dy}x_{2}-(\phi+\mu)x_{2}+(\beta_1x_{2}+\beta_2 x_{4}) x_1, & y\in\R,\cr
 0=D_v\frac{d^2}{•dy^2} x_4-c\frac{d}{•dy}x_{4}-\eta x_{4}+\beta x_{2} x_{3}, & y\in\R.
\end{cases}
$$
Consider the sequence
$$
(x_2^n(y),x_4^n(y))^T=\left(\frac{x_2(y-n)}{x_{2}(-n)},\frac{x_4(y-n)}{x_{2}(-n)}\right)^T.
$$
Hence, $ (x_2^n(y),x_4^n(y))^T$ satisfies
$$
\begin{cases}
 0=D_h\frac{d^2x^n_{2}}{dy^2}-c\frac{dx^n_{2}}{dy}-(\phi+\mu)x^n_{2}+(\beta_1x^n_{2}+\beta_2 x^n_{4}) x_1(\cdot-n), & y\in\R,\cr
 0=D_v\frac{d^2x^n_{4}}{•dy^2}-c\frac{d x^n_{4}}{•dy}-\eta x^n_{4}+\beta x^n_{2} x_{3}(\cdot-n), & y\in\R.
\end{cases}
$$
By Lemmas \ref{lemma11} and \ref{lemma12} there is $M_{|c|}>1$ such that
$$
\max\left\{ x_2^n(y), \left|\frac{d x_2^n(y)}{dy}\right|\right\}\le M_{|c|}\frac{x_2(y-n)}{x_2(-n)}\leq M_{|c|}\frac{x_{2}(-n)}{x_{2}(-n)}e^{|(y-n)-n|M_{|c|}}=M_{|c|}e^{|y|M_{|c|}}.
$$
Hence, by estimates for parabolic equations
we may suppose that   $ (x_2^n(y),x_4^n(y))^T\to (x_2^*(y),x_4^*(y)) $ in $C^2_{\rm loc}(\mathbb{R}^2)$ and $(x_2^*(y),x_4^*(y))$ satisfies
\begin{equation}\label{qq00001}
\begin{cases}
 0=D_h\frac{d^2x^*_{2}}{dy^2}-c\frac{dx^*_{2}}{dy}-l_0x^*_{2}+\frac{\beta_2 b_1}{\mu} x^*_{4}, & y\in\R\cr
 0=D_v\frac{d^2x^*_{4}}{•dy^2}-c\frac{d x^*_{4}}{•dy}-\eta x^*_{4}+\frac{\beta b_2}{\eta} x^*_{2}, & y\in\R\cr
 x^*_2(y)>0, \ x_4^*(y)>0, \cr
 x^*_2(0)=1, \ \frac{1}{m_c}\leq \frac{x^*_2(y)}{x^*_{4}(y)}\leq m_c,\quad \forall\ y\in\mathbb{R},
\end{cases}
\end{equation}
where $m_c$ is given by Lemma \ref{lemma12}. To complete the proof of the theorem, we will show that either $x^*_2(y)$ changes sign or $x^*_4(y)$ changes sign. Observe that
$$
\frac{d}{dy}\left(\begin{array}{c}
x_2^*\cr
x_4^*\cr
\dot{x}^*_2\cr
\dot{x}_4^*
\end{array} \right)=
\underbrace{\left[
\begin{array}{cccc}
0 & 0& 1&0\cr
0 & 0& 0&1\cr
\frac{l_0}{D_h}& -\frac{\beta_2b_1}{\mu D_h}&\frac{c}{D_h}&0\cr
-\frac{\beta b_2}{\eta D_v} & \frac{\eta}{D_v}&0 &\frac{c}{D_v}
\end{array}
\right]}_{=\mathbb{M}}\left(\begin{array}{c}
x_2^*\cr
x_4^* \cr
\dot{x}^*_2\cr
\dot{x}_4^*
\end{array} \right).
$$

The characteristic polynomial of  $\mathbb{M}$ is
$$
P(\lambda)=\lambda^2\left(\frac{c}{D_h}-\lambda\right)\left( \frac{c}{D_v}-\lambda\right)+\frac{\eta}{D_v}\left(\frac{c}{D_h}-\lambda\right)\lambda+ \frac{l_0}{D_h}\left(\frac{c}{D_v}-\lambda\right) \lambda -\frac{1}{D_vD_h}(l_1-l_0\eta).
$$
Hence, since $\mathcal{R}_0>1$, then $l_1>\eta l_0$ and
 $$ P(0)=-\frac{1}{D_vD_h}(l_1-l_0\eta)<0.$$
Thus, since $P(\pm \infty)=+\infty$, we conclude that the matrix $\mathbb{M}$  has at least two real eigenvalues of opposite signs.  Observe that if $\lambda\in\mathbb{C}$ is an eigenvalue of the matrix $\mathbb{M}$ with an associated eigenvector $(k_2,k_4,\tilde{k}_2,\tilde{k}_4)^T$ then $(\tilde{k}_2,\tilde{k}_4)^T=(\lambda k_2,\lambda k_4)^T$ and
$$
\underbrace{\left[
\begin{array}{cc}
D_h\lambda^2-l_0      &  \frac{\beta_2b_1}{\mu}\cr
\frac{\beta b_2}{\eta } & D_v\lambda^2-\eta
\end{array}
\right]}_{=\mathcal{M}(\lambda)}\left( \begin{array}{c}
k_2\cr
k_4
\end{array}\right)=\lambda c\left( \begin{array}{c}
k_2\cr
k_4
\end{array}\right).
$$
Therefore, we have that $\lambda c$ is an eigenvalue  of the matrix $ \mathcal{M}(\lambda) $. In particular, if $\lambda\in\mathbb{R}$ is a real eigenvalue of $\mathbb{M}$, it follows from the results of Section \ref{Sec1} that
$$
\lambda c\in\{\alpha_{\max}(\lambda),\alpha_{\min}(\lambda)\}.
$$
Hence, since $\alpha_{\max}(\lambda)>\alpha_{\min}(\lambda)$ for every $\lambda\in\mathbb{R}$, we conclude that the dimension, ${\text Dim}(\mathbb{S}_{\lambda}), $ of the eigenspace associated to any real eigenvalue $\lambda$ of the  matrix $\mathbb{M}$ is always one. That is
$$
{\text Dim}(\mathbb{S}_{\lambda})=1.
$$

We claim that it is always the case that
\begin{equation}\label{gggggg1}
\lambda c=\alpha_{\min}(\lambda).
\end{equation}
Recall that $\alpha_{\max}(\lambda)>0$ since $\mathcal{R}_0>1$.  Therefore, if $\lambda c=\alpha_{\max}(\lambda)$, then $\lambda$ and $c$ have same sign. So, we have two cases.

If $\lambda>0$ and $c>0$ then $c=\frac{\alpha_{\max}(\lambda)}{\lambda}\geq c^*$, which is impossible since $c<c^*$. Hence $ \lambda c=\alpha_{\min}(\lambda).$

If $\lambda<0$ and $c<0$ then $-c=\frac{\alpha_{\max}(|\lambda|)}{|\lambda|}\geq c^*$. So, $|c|\geq c^*$, which is impossible. Hence $ \lambda c=\alpha_{\min}(\lambda).$

\noindent In view of \eqref{gggggg1} and \eqref{s-eq10}, we  know that
$$
\mathbb{S}_{\lambda}={\text Span}\left\{\left(\begin{array}{c}
1\cr
\delta_{\lambda}\cr
\lambda\cr
\lambda\delta_{\lambda}
\end{array}\right)\right\},
$$
where $\delta_{\lambda}=\frac{(\alpha_{\min}(\lambda)-m_1(\lambda))\mu}{\beta_2b_1}<0$ (see Lemma \ref{lemma0} (ii)). It is convenient to set $E_{\lambda}:=(1,\delta_{\lambda})^T$.

To complete the proof, we distinguish two cases, based on the eigenvalues of $\mathbb{M}$.

{\bf Case 1.} The matrix $\mathbb{M}$ has purely complex eigenvalues. Since  $\mathbb{M}$ also has two real eigenvalues of opposite signs $\lambda_1<0<\lambda_2$, we may suppose that $\lambda=a\pm ib$ with $b>0$, are two remaining complex roots of $P(\lambda)$. In this case, $(x^*_2(y),x^*_4(y))^T$ can be written as

\begin{align*}
\left(\begin{array}{c}
x^*_2(y)\cr
x^*_4(y)\cr
\end{array}\right)=&\sum_{i=1}^2a_ie^{\lambda_iy}E_{\lambda_i}+ a_3e^{ay}\left(\begin{array}{c}
\cos(by)\cr
a\cos(by)-b\sin(by)
\end{array}\right)+ a_4e^{ay}\left(\begin{array}{c}
\sin(by)\cr
b\cos(by)+a\sin(by)
\end{array}\right),
\end{align*}
where $a_i\in\mathbb{R}$, $i=1,\cdots,4$ are not all equal to zero and are uniquely determined. So
\begin{equation}\label{q1}
x^*_2(y)=a_1e^{\lambda_1y}+a_2e^{\lambda_2y}+e^{ay}f(y),
\end{equation}
where $f(y)=a_3\cos(by)+a_4\sin(by)$ and
\begin{equation}\label{q2}
x^*_4(y)=a_1\delta_{\lambda_1}e^{\lambda_1y}+a_2\delta_{\lambda_2}e^{\lambda_2y}+e^{ay}g(y)
\end{equation}
with 
$$ g(y)=(a_3(a\cos(by)-b\sin(by))+a_4(b\cos(by)+a\sin(by))).$$
 Clearly, from \eqref{q1} we see that if $f(y)\equiv 0$, then since $x^*_2(y)> 0$ for every $y\in\mathbb{R}$ and $\lambda_1<0<\lambda_2$, we must have that  $a_1\geq 0$ and $a_2\ge0$, which yield that  $x^*_4(y)\le0$ for every $y\in\mathbb{R}$ since $\delta_{\lambda_i}<0$, $i=1,2$, contradicting \eqref{qq00001}. So, $ f(y)$ changes sign when $|y|\gg 1$. Similarly, from \eqref{q2} we see that if $g(y)\equiv 0$, since $x^*_4(y)>0$ for every $y\in\mathbb{R}$  and $\delta_{\lambda_i}<0$, $i=1,2$, then we have that $a_i\leq 0$ for each $i=1,2$, which implies that $ x^*_2(y)<0$ whenever $f(y)<0$.  So, both $f(y)$ and $g(y)$ change sign when $|y|\gg 1$. These in turn imply that  $a_1\neq 0$ or  $a_2\neq 0$ since $x^*_2(y)>0$ and $x^*_4(y)>0$ for every $y\in\mathbb{R}$.

Next, if $a=\lambda_{i_0}$ for some $i_0\in\{1,2\}$, thus since $x^*_2(y)>0$, we must have that $a_{i_0}\geq \|f\|_{\infty}>0$, and since $x^*_{4}(y)>0$ for every $y\in\mathbb{R}$, we must also have that $a_{i_0}\delta_{\lambda_{i_0}}\ge \|g\|_{\infty}>0$. Hence since $\delta_{\lambda_{i_0}}<0$, we obtain that $a_{i_0}<0<a_{i_0}$, which is impossible. Thus $a\notin\{\lambda_1,\lambda_2\}$.  Therefore, since $f(y)$ changes sign when $|y|\gg 1$, to guarantee that $x^*_{2}(y)>0$ for $|y|\gg 1$, we must have that $\lambda_1<a<\lambda_2$ and  $a_{1}\gg 1$ and $a_{2}\gg 1$,  which imply that $ x^*_4(y)<0$ when $|y|\gg 1$ since $\delta_{\lambda_1}<0$ and $\delta_{\lambda_2}<0$.  So, in this case we must have that either $x^*_2(y)$ changes sign or $x^*_{4}(y)$ changes sign.

{\bf Case 2.} The matrix $\mathbb{M}$ has no complex eigenvalues.

$(i)$ If all the eigenvalues are distinct.  Similar arguments as in the above yield that the corresponding eigenvectors can be chosen such that $(k_2^i,k_4^i)^T=E_{\lambda_i}=(1,\delta_{\lambda_i})^T$, $i=1,2,3,4$. Hence
$$
\left(\begin{array}{c}
x^*_2(y)\cr
x^*_4(y)\cr
\end{array}\right)=\sum_{i=1}^4a_ie^{\lambda_iy }E_{\lambda_i}=\left( \begin{array}{c}
\sum_{i=1}^4a_ie^{\lambda_i y}\cr
\sum_{i=1}^4a_i\delta_{\lambda_i}e^{\lambda_i y}
\end{array}\right)
$$
where $a_i\in\mathbb{R}$, $i=1,\cdots,4$ are not all equal to zero. Since $\lambda_i\neq \lambda_j$ for $i\neq j$,  $\delta_{\lambda_i}<0$ for every $i\in\{1,\cdots,4\}$, and $\min_{i}\{\lambda_i\}<0<\max_i\{\lambda_i\}$;  then  either $x_2^*(y_0)<0$  or $x^*_4(y_0)<0$  for some  $y_0\in\mathbb{R}$, which contradicts \eqref{qq00001}.

$(ii)$ If there is a double eigenvalue. Since $P(0)<0$, then we can't have two double roots. So, we are left with the only possibility
$$
P(\lambda)=(\lambda-\lambda_3)^2(\lambda-\lambda_1)(\lambda-\lambda_2),
$$
with $\lambda_1<0<\lambda_2$, and $\lambda_3\notin\{\lambda_1,\lambda_2\}.$  Since ${\text Dim}(\mathbb{S}_{\lambda_3})=1$, hence
$$
\left(\begin{array}{c}
x^*_2(y)\cr
x^*_4(y)\cr
\end{array}\right)=\sum_{i=1}^2a_ie^{\lambda_iy } E_{\lambda_i}+e^{\lambda_3y}\left[ a_3E_{\lambda_3}+a_4(yE_{\lambda_3}+E_{\lambda_3}^1)\right]
$$
where $a_i\in\mathbb{R}$, $i=1,\cdots,4$ are not all equal to zero, and  $\{(E_{\lambda_3},\tilde{E}_{\lambda_3})^T,(E_{\lambda_3}^1,\tilde{E}_{\lambda_3}^1)^T\}$ form a set of two linearly independent  generalized eigenvectors of $\lambda_3$.

If $a_4=0$, then similar arguments as in $(i)$ yield that  either $x_2^*(y_0)<0$  or $x^*_4(y_0)<0$  for some  $y_0\in\mathbb{R}$, which contradicts \eqref{qq00001}.

If $a_4\neq 0$ and $\lambda_3\geq \lambda_2$, then we must have that  $a_4\ge0$ and $a_1>0$ to ensure that $x^*_2(y)>0$ for $y\gg 1$ and $y\ll -1$, respectively. Whence, $x^*_4(y)<0$ for $y\ll -1$ since $a_1\delta_{\lambda_1}<0$, which contradicts \eqref{qq00001}.

If $a_4\neq 0$ and $\lambda_3\leq \lambda_1$, then we must have that  $a_4\le0$ and $a_2>0$ to ensure that $x^*_2(y)>0$ for $y\ll- 1$ and $y\gg 1$, respectively. Hence $x^*_4(y)<0$ for $y\gg 1$ since $a_2\delta_{\lambda_2}<0$, which contradicts \eqref{qq00001}.

If $a_4\neq 0$ and $\lambda_1< \lambda_3<\lambda_2$, then we must have that  $a_1\ge 0$, and $a_2\geq 0$, and $a_2+a_1>0$ to ensure that $x^*_2(y)>0$ for $|y|\gg 1$.  Hence $x^*_4(y)<0$ for some $y\gg 1$ since $\delta_{\lambda _i}<0$, $i=1,2$, contradicting \eqref{qq00001}.

$(iii)$ If there is an eigenvalue with multiplicity three, say $ \lambda_1 $. Then
$$
P(\lambda)=(\lambda-\lambda_1)^3(\lambda-\lambda_4)
$$
with $\lambda_1\lambda_4<0$. Since ${\text Dim}(\mathbb{S}_{\lambda_1})=1$, hence
$$
\left(\begin{array}{c}
x^*_2(y)\cr
x^*_4(y)\cr
\end{array}\right)=a_4e^{\lambda_4y } E_{\lambda_4}+e^{\lambda_1y}\left[ a_2E_{\lambda_1}+a_3\left(yE_{\lambda_1}+E_{\lambda_1}^1\right)+a_4\left(\frac{y^2}{2}E_{\lambda_1}+yE_{\lambda_1}^1+E_{\lambda_1}^2\right)\right]
$$
where $a_i\in\mathbb{R}$, $i=1,\cdots,4$ are not all equal to zero, and $\{(E_{\lambda_1},\tilde{E}_{\lambda_1})^T,(E_{\lambda_1}^1,\tilde{E}_{\lambda_1}^1)^T,(E_{\lambda_1}^2,\tilde{E}_{\lambda_3}^2)^T\}$ form a set of three linearly independent  generalized eigenvectors of $\lambda_1$. Since $\lambda_1\neq \lambda_2$ and $\delta_{\lambda_i}<0$ for every $i\in\{1,2\}$,  similar arguments as in the above yield that  either $x_2^*(y_0)<0$  or $x^*_4(y_0)<0$  for some  $y_0\in\mathbb{R}$, which contradicts \eqref{qq00001}.

Therefore, we conclude that there is no traveling wave solution with speed $|c|<c^*$.

\end{proof}

For the rest of this section, we suppose that $D_h=D_v = D$. Before presenting the proof of  Theorem \ref{main-tm-2}(ii), we first recall the following quantities
$$ l_0=\mu+\phi-\frac{\beta_1b_1}{\mu},\quad l_1=\frac{\beta\beta_2b_1b_2}{\eta\mu}, $$
and
$$
\alpha_{max}(0)=\frac{1}{2}\left( -(\eta+l_0)+\sqrt{(\eta-l_0)^2+4l_1}\right).
$$
 Hence it readily follows that
\begin{equation}\label{ss-ss-s20}
\alpha:=\frac{\alpha_{\max}(0)+l_0}{\frac{\beta_2b_1}{\mu}}=\frac{\sqrt{(\eta-l_0)^2+4l_1}-(\eta-l_0)}{2\frac{\beta_2b_1}{\mu}}>0.
\end{equation}
We need the following lemma.

\begin{lem}\label{lemma9} Suppose $\alpha_{\max}(0)>0$. Let $\alpha$ be given by \eqref{ss-ss-s20}.  For every $0<\varepsilon<\frac{\alpha_{\max}(0)}{(1+\alpha)}$ and $0<\gamma<\sqrt{4D(\alpha_{\max}(0)-\varepsilon(1+\alpha))}$, the function
$$
h(y)=e^{\frac{\gamma}{2D}y}\cos\left(\frac{\tilde{\gamma}}{2D}y\right),\quad -L<y<L
$$
where $\tilde{\gamma}=\sqrt{4D(\alpha_{\max}(0)-\varepsilon(1+\alpha))-\gamma^2}$ and $L=\frac{D\pi}{\tilde{\gamma}}$ satisfies $h\in C^{2}(-L,L)\cap C[-L,L]$ and
\begin{equation}\label{ss-ss-s21}
0=Dh''-\gamma h'+ \left( -(l_0+\varepsilon)+\left(\frac{\beta_2b_1}{\mu}-\varepsilon\right)\alpha\right)h,\quad h(\pm L)=0, \quad h(y)>0,\ \forall -L<y<L.
\end{equation}
\end{lem}
\begin{proof}
Note from \eqref{ss-ss-s20} that
$$
 -(l_0+\varepsilon)+\left(\frac{\beta_2b_1}{\mu}-\varepsilon\right)\alpha=\frac{\beta_2b_1}{\mu}\alpha-l_0-\varepsilon(1+\alpha)=\alpha_{\max}(0)-\varepsilon(1+\alpha)
$$
and that the complex number $r=\frac{1}{2D}\left(\gamma+i\tilde{\gamma}\right)$ is a root of the quadratic equation
$$
Dr^2-\gamma r+(\alpha_{\max}(0)-\varepsilon(1+\alpha))=0.
$$
Hence the real part $h(y)$ of $e^{ry}$ satisfies \eqref{ss-ss-s21}. The remaining assertion of the lemma easily follows from the definition of $L$ and properties of trigonometric functions.
\end{proof}

\begin{proof}[Proof of Theorem \ref{main-tm-2}(ii)] Suppose to the contrary that ${\bf x}(t,y)={\bf x}(y+ct)$ is a traveling wave solution of \eqref{pde1}  connecting $E_0$ at one end with speed $c<c^*$. Recall from Remark \ref{remark1} (ii) that  $c^*=2\sqrt{D\alpha_{\max}(0))}$.  So we can choose $0<\varepsilon\ll \min\left\{\frac{\beta_2b_1}{\mu},\frac{\beta b_2}{\eta(1+\alpha)\alpha},1\right\}$  and $\gamma>0$ such that
$$
\max\{c,0\}<\gamma<\sqrt{4D(\alpha_{\max}(0)-\varepsilon(1+\alpha))}.
$$
Next, consider the function
$$
h(y)=e^{\frac{\gamma}{2D}y}\cos\left(\frac{\tilde{\gamma}}{2D}y\right),\quad -L<y<L
$$
of Lemma \ref{lemma9} and define
$$
(u_2(y),u_{4}(y))^T:=(h(y),\alpha h(y))^T,\quad -L<y<L.
$$
Now, observe  from \eqref{ss-ss-s20} that $\alpha$ is the positive root of
 $$
\alpha^2\frac{\beta_2b_1}{\mu} +(\eta-l_0)\alpha-\frac{\beta b_2}{\eta}=0.
 $$
 Whence,  taking $\varepsilon_1=\varepsilon(1+\alpha)\alpha$ and dividing both sides of the last equation by $\alpha$, it holds that
$$
-(l_0+\varepsilon) +\left(\frac{\beta_2b_1}{\mu}-\varepsilon\right)\alpha=-\eta  +\left(\frac{\beta b_2}{\eta}-\varepsilon_1\right)\frac{1}{\alpha}.
$$
Therefore, \eqref{ss-ss-s21} implies that
$$
\begin{cases}
0=D\Delta u_2-\gamma\partial_yu_2 -(l_0+\varepsilon)u_2+\left(\frac{\beta_2b_1}{\mu}-\varepsilon\right)u_4\cr
0=D\Delta u_4-\gamma\partial_yu_4 -\eta u_4 +\left(\frac{\beta b_2}{\eta}-\varepsilon_1\right)u_2.
\end{cases}
$$
Hence,   the function
\begin{equation}\label{ss-ss-s22}
(\underline{x}_2(t,y),\underline{x}_{4}(t,y))^T=(u_2(y+\gamma t),u_{4}(y+\gamma t))^T
\end{equation}
satisfies
\begin{equation}\label{ss-ss-s23}
\begin{cases}
\partial_t\underline{x}_2=D\Delta_y \underline{x}_2 -(l_0+\varepsilon)u_2+\left(\frac{\beta_2b_1}{\mu}-\varepsilon\right)u_4,\quad  &-\gamma t-L<y<L-\gamma t,\cr
\partial_t\underline{x}_4=D\Delta_y \underline{x}_4 -\eta \underline{x}_4 +\left(\frac{\beta b_2}{\eta}-\varepsilon(1+\alpha)\alpha\right)\underline{x}_2,\quad & -\gamma t-L<y<L-\gamma t.
\end{cases}
\end{equation}
But, there is $Y_{\varepsilon}\ll- 1$ such that
$$
 \beta_1{ x}_1(y)\geq \frac{\beta_1b_1}{\mu}-\varepsilon, \quad \beta_2{ x}_1(y)\geq \frac{\beta_2b_1}{\mu}-\varepsilon, \text{and}\quad \beta{ x}_{3}(y)\geq \frac{\beta b_2}{\mu}-\varepsilon_1, \quad \forall y\le Y_{\varepsilon}.
$$
Hence, ${\bf x}(t,y)={\bf x}(y+ct)$ satisfies
\begin{equation}\label{ss-ss-s24}
\begin{cases}
 \frac{\partial x_2}{\partial t}\geq D\Delta_y x_2-(l_0+\varepsilon)x_2+\left(\frac{\beta_2b_1}{\mu}-\varepsilon\right) x_4,& y\le Y_{\varepsilon}-ct,\ t\in\R,\cr
 \frac{\partial x_4}{\partial t}\geq D\Delta_y x_4-\eta x_4+\left(\frac{\beta b_2}{\eta}-\varepsilon_1\right) x_2,& y\le Y_\varepsilon-ct, \ t\in\R.
\end{cases}
\end{equation}
Taking $T_{\varepsilon}=\frac{L-Y_{\varepsilon}}{\gamma-c}$, we note that
$\gamma t-L\geq ct +Y_\varepsilon$ whenever $t\geq T_{\varepsilon
}$. Note also that we can choose $0<\sigma_{\varepsilon}\ll 1$ such that
$$
(\sigma_\varepsilon\underline{x}_2(T_{\varepsilon},y),\sigma_\varepsilon\underline{x}_{4}(T_\varepsilon,y))^T\leq ({x}_2(T_\varepsilon,y),{x}_{4}(T_\varepsilon,y))^T, \quad -\gamma T_\varepsilon-L<y<-\gamma T_{\varepsilon}+L.
$$
 With $\sigma_{\varepsilon}$ chosen, it holds that
$$
(\sigma_\varepsilon\underline{x}_2(t,-\gamma t\pm L),\sigma_\varepsilon\underline{x}_{4}(t,\gamma t \pm L))^T=(0,0)^T\leq ({x}_2(t,\gamma t \pm L),{x}_{4}(t,\gamma t\pm L))^T, \quad t\geq T_{\varepsilon}.
$$
Therefore, by \eqref{ss-ss-s23}, \eqref{ss-ss-s24} and the comparison principle for cooperative systems, we conclude that
$$
(\sigma_\varepsilon\underline{x}_2(t,y),\sigma_\varepsilon\underline{x}_{4}(T_\varepsilon,y))^T\leq ({x}_2(t,y),{x}_{4}(T_\varepsilon,y))^T, \quad -\gamma t-L<y<-\gamma t+L,\ t> T_{\varepsilon}.
$$
In particular, for $y=-\gamma t+\frac{L}{2}$ in the last inequality and recalling \eqref{ss-ss-s22} and the explicit expression of $h(y)$, we conclude that
$$
0<\sigma_{\varepsilon}u_{2}\left(\frac{L}{2}\right)\leq {\bf x}_{2}\left((c-\gamma)t+\frac{L}{2}\right)\to 0\ \text{as}\ t\to\infty
$$
since $\gamma>c$, which is impossible. Therefore, we must have that $c\geq c^*$.

\end{proof}

\section{Numerical simulation}

In this section, we provide some numerical simulations to support our theoretical results. Since the spatial and temporal intervals are infinite, we will consider the interval $[0, 500]$ in space and the interval $[0,50]$ in time for the sake of illustration.

\vspace{0.2in}
\subsection{ Initial and boundary conditions}

To system~\eqref{pde1}, we assign the following piecewise continuous functions as initial conditions
\begin{align*}
x_1(0,y)=
\begin{cases}
x_1^{**} &\quad\text{if $0\le y< 200$,}\\
x_1^0 &\quad\text{if $200\le y\le 500$},
\end{cases}
&\qquad
x_2(0,y)=
\begin{cases}
x_2^{**} &\quad\text{if $0\le y< 200$,}\\
x_2^0 &\quad\text{if $200\le y\le 500$},
\end{cases}
\end{align*}

\begin{align*}
x_3(0,y)=
\begin{cases}
x_3^{**} &\quad\text{if $0\le y< 200$,}\\
x_3^0 &\quad\text{if $200\le y\le 500$},
\end{cases}
&\qquad
x_4(0,y)=
\begin{cases}
x_4^{**} &\quad\text{if $0\le y< 200$,}\\
x_4^0 &\quad\text{if $200\le y\le 500$},
\end{cases}
\end{align*}
where $E_0 = (x_1^0, x_2^0, x_3^0, x_4^0)^T = \left(\frac{b_1}{\mu}, 0, \frac{b_2}{\eta}, 0\right)^T$ is the disease-free equilibrium and $E_1 = (x_1^{**}, x_2^{**}, x_3^{**}, x_4^{**})^T$ is the endemic equilibrium given by ~\eqref{equili1}--\eqref{equili2}. Note that the initial conditions are chosen as such so that the solution has a wave-like shape. Figure~\ref{fig:Initial_condition7-12} displays the above initial conditions with the susceptible and infected hosts at the left and the susceptible and infected vectors at the right.
\begin{figure}[h]
\centering
\includegraphics[width = 1.1\textwidth, height = .43\textheight]{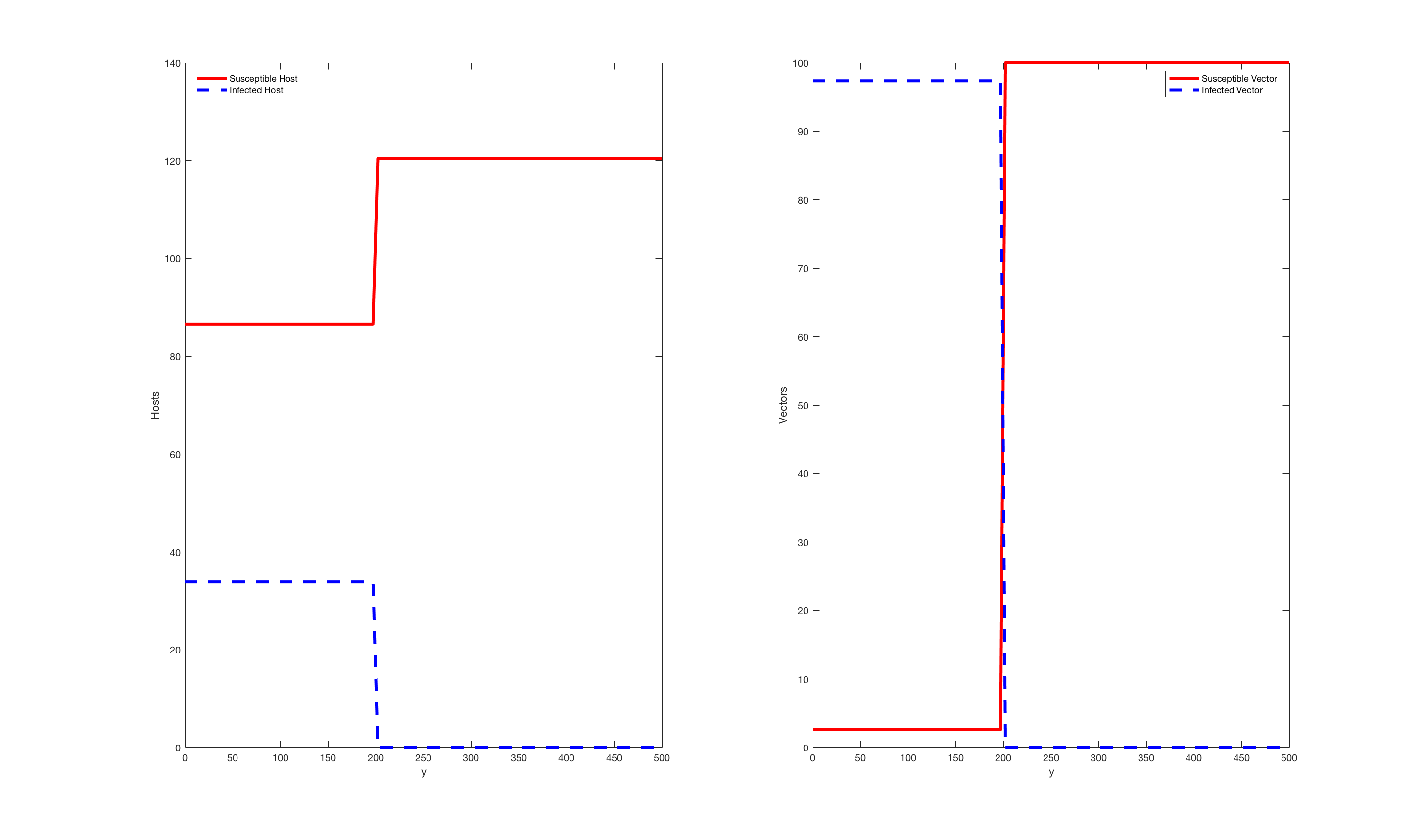}
\caption{Graphs of initial conditions suggesting a wave-like shape for the traveling wave solution.}
\label{fig:Initial_condition7-12}
\end{figure}

On the other hand, the fact that the model assumes no death due to the disease and no birth in both species implies that the population is self-contained within the given region for all time. In other words, there is no population flux. This suggests using the homogeneous Neumann boundary conditions on the boundary of the spatial domain at all times.
To that end, at the left at $y = 0$ and at the right at $y = 500$, we use the following homogeneous Neumann boundary conditions
\begin{align*}
\begin{cases}
\frac{\partial x_1}{\partial y}(0,t) = \frac{\partial x_1}{\partial y}(500,t) = 0, &\quad 0\le t\le 50 \\
\frac{\partial x_2}{\partial y}(0,t) = \frac{\partial x_2}{\partial y}(500,t) = 0, &\quad 0\le t\le 50 \\
\frac{\partial x_3}{\partial y}(0,t) = \frac{\partial x_3}{\partial y}(500,t) = 0, &\quad 0\le t\le 50 \\
\frac{\partial x_4}{\partial y}(0,t) = \frac{\partial x_4}{\partial y}(500,t) = 0, &\quad 0\le t\le 50.
\end{cases}
\end{align*}

\vspace{0.2in}
\subsection{ Numerical experiments and results}

In this subsection, we perform experiments to help visualize the theoretical results. We use the programming software MATLAB to conduct the numerical simulations. For the mesh, we take 100 linearly spaced spatial points between 0 and 500 and 20 linearly spaced temporal points between 0 and $50$.

The parameter values used for the numerical solutions are given in Table~\ref{table:parametervalues}.
\begin{table}[ht]
\caption{Parameter values for determining the reproductive number $\mathcal{R}_0$}  
\label{table:parametervalues}
\centering
\begin{tabular}{ccccccccc}
\hline\hline
Parameter & $\beta_1$ & $\beta_2$ & $\eta$ & $\beta$ & $\phi$ & $\mu$ & $b_1$ & $b_2$ \\ [1ex] \hline
Value        & 0.005         &  0.003      &  0.001 &  0.0011 &  0.35    &  0.83  &  100    &  0.1 \\ [1ex] \hline
\end{tabular}
\end{table}

We note that the description of each of the parameters in Table~\ref{table:parametervalues} have been used with different values to study the numerical simulations of a stochastic epidemic model of vector-borne diseases with direct mode of transmission applied to malaria. See for instance \cite{jovanovic2012stochastically}.

The parameter values in Table~\ref{table:parametervalues} yield the reproductive number $\mathcal{R}_0 = 34.20$, the endemic equilibrium $E_1 = (86.60, 33.87, 2.61, 97.38)$, and the disease-free equilibrium $E_0 = (120.48, 0, 100, 0)$.
In addition to these values, the diffusion rates we use for the hosts and vectors are $D_h= 0.2$ and $D_v= 0.5$, respectively. Furthermore, note that $\phi = 0.35 < \frac{\beta_1b_1}{\mu} = 0.60$. Consequently, Theorem~\ref{main-tm-02} guarantees the existence of a traveling wave solution for system~\eqref{pde1} with speed $c^*$ connecting $E_0$ and $E_1$. We illustrate the surface plots of this solution in Figure~\ref{fig:surfaces}.

One way to check for efficiency of the experiment is to obtain the approximate solution at the final time. To illustrate, in Figure~\ref{fig:profiles} we sketch the profiles of the solution at the final time. In can be clearly seen that analogous to the surface plots, these profiles indeed connects $E_0$ to $E_1$. At the same time, these profiles show that the traveling wave solution of system~\eqref{pde1} is not monotone.

Meanwhile, Lemma~\ref{lemma1} shows that the minimal wave speed $c^*$ is the minimum value of the function $c_{\lambda}$ given by Equation~\eqref{s-eq6} and achieved at the unique number $\lambda^*$. To illustrate, we plot in Figure~\ref{fig:clambda} the function $c_{\lambda}$ together with the minimal wave speed of $c^* = 0.3410$ achieved at $\lambda^* = 0.3583$.

\begin{figure}[thbp]
\centering
\includegraphics[width = 1.1\textwidth, height = .43\textheight]{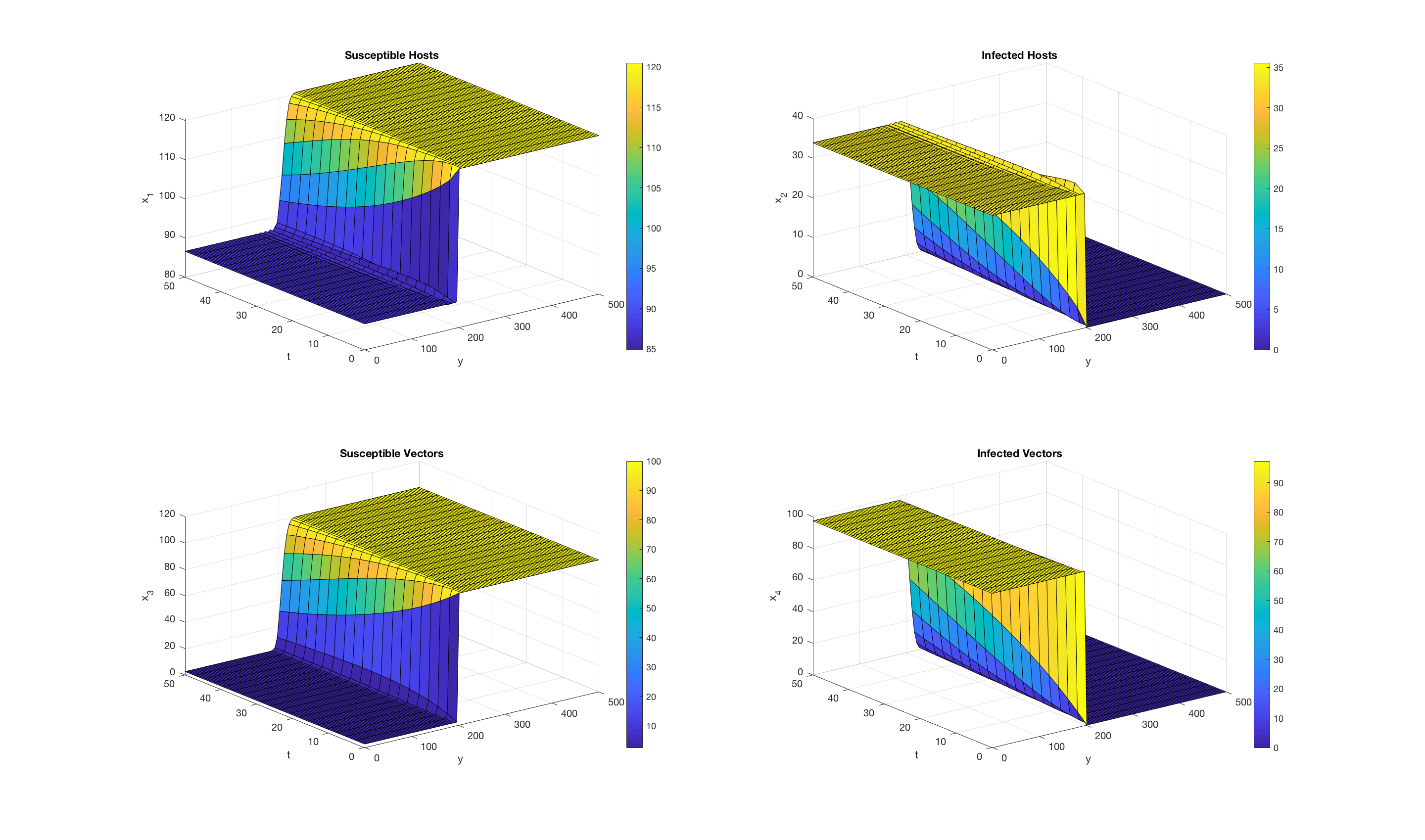}
\caption{Numerical approximation illustrating the existence of a traveling wave solution to system~\eqref{pde1} with minimal speed $c^*$}
\label{fig:surfaces}
\end{figure}

\begin{figure}[bhtp]
\centering
\includegraphics[width = 1.1\textwidth, height = .43\textheight]{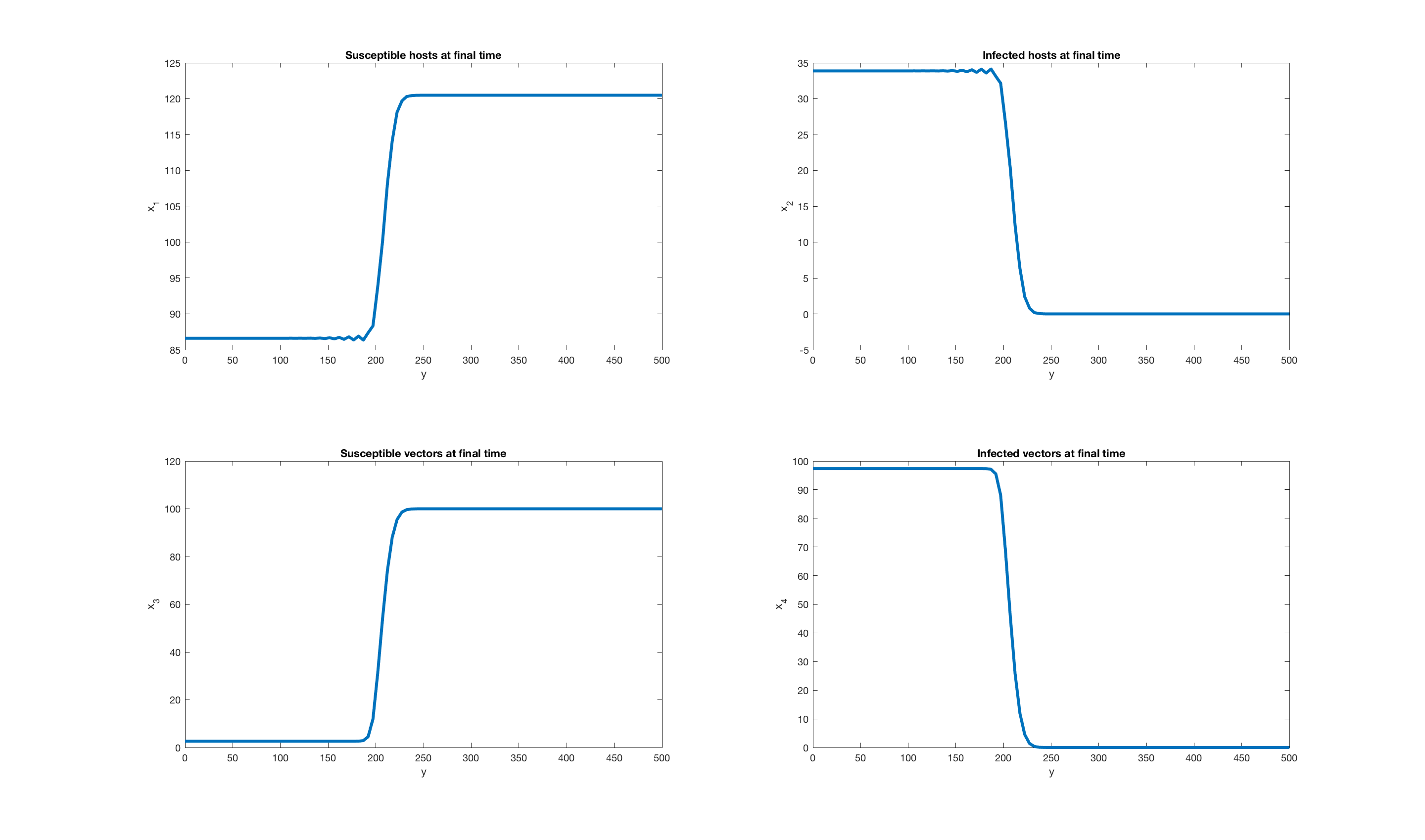}
\caption{Profiles of the traveling wave solution of~\eqref{pde1} connecting $E_0$ and $E_1$ at the final time $T = 50$.}
\label{fig:profiles}
\end{figure}

\begin{figure}[thbp]
\centering
\includegraphics[width = 1.1\textwidth, height = .43\textheight]{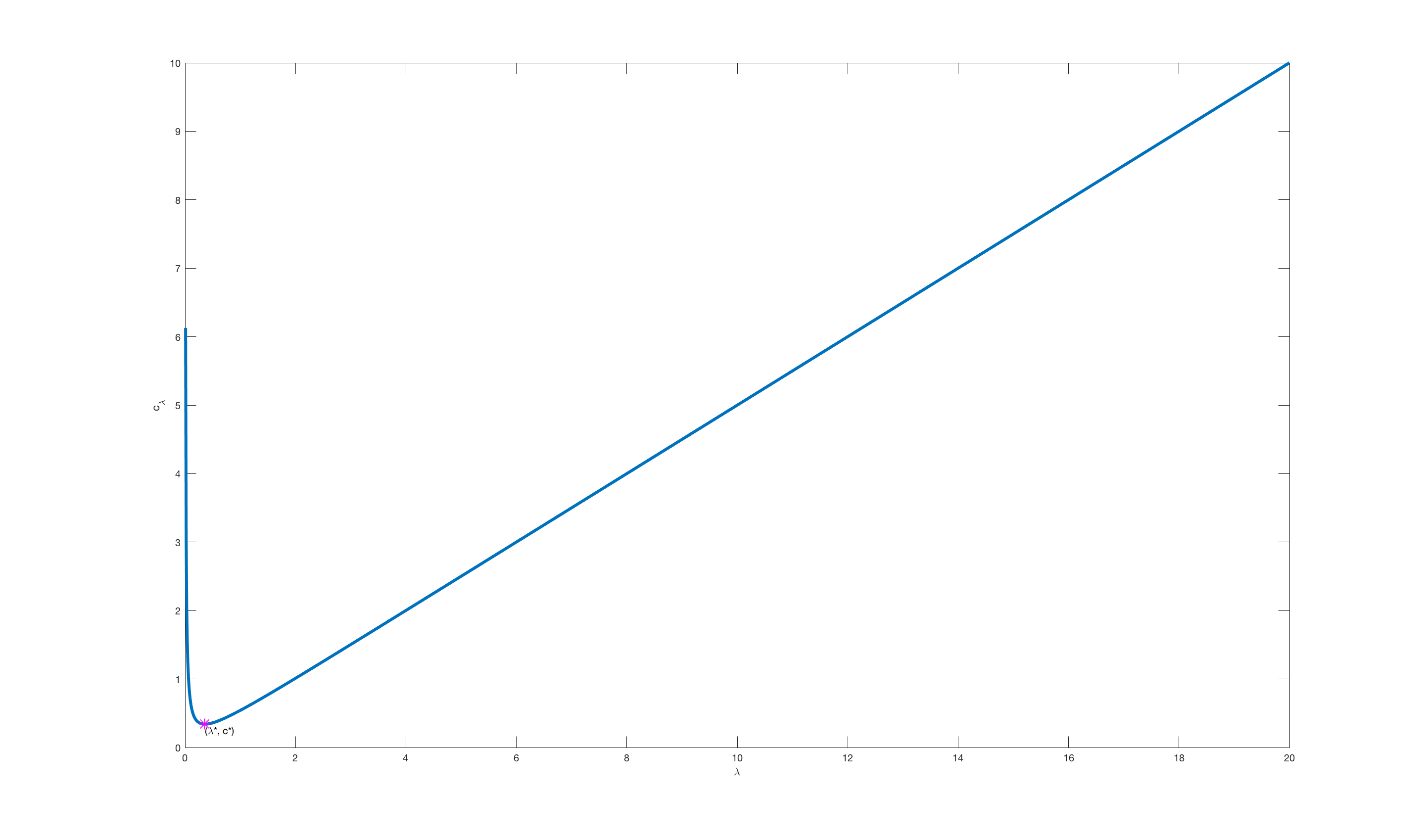}
\caption{Graph of the function $c_{\lambda}$ in ~\eqref{s-eq6} illustrating the existence and uniqueness of the minimal wave speed $c^*$}
\label{fig:clambda}
\end{figure}

\section{Discussion/Conclusion}
We investigate the existence and non-existence of traveling wave solutions for a parabolic epidemic model with direct transmission given by system~\eqref{pde1} assuming that the diffusion rates of the susceptible hosts $d_1$ and the infected hosts $d_2$ are the same $d_1 = d_2 = D_h$ and those of the susceptible vectors $d_3$ and the infected vectors $d_4$ are the same $d_3 = d_4 = D_v$, and $\phi \le \frac{\beta_1b_1}{\mu}$. When the reproductive number $\mathcal{R}_0>1$ we show that there is a minimum wave speed $c^*$ such that the parabolic system admits traveling wave solutions with speed $c$ for any $0<c^*\le c$ connecting the disease-free equilibrium $E_0$ and the endemic equilibrium $E_1$. For $\mathcal{R}_0<1$ or $\mathcal{R}_0>1$ and $0<c<c^*$ we prove that the system has no nontrivial nonnegative traveling wave solutions connecting the two equilibria. Moreover, we provide numerical simulations to illustrate the existence of the approximate solutions to the traveling wave solutions of the parabolic epidemic system.

On the other hand, it would be of great mathematical interest to study the existence of traveling wave solutions of \eqref{pde1} connecting $E_0$ and $E_1$ if $\mathcal{R}_0>1$ and  $\phi>\frac{\beta_1b_1}{\mu}$. In particular, it would be interesting to know whether system~\eqref{pde1} has a minimal wave speed in this case. Another direction would be to study the existence of traveling wave solutions of system~\eqref{pde1} if $d_1\ne d_2$ and/or $d_3\ne d_4$.  Finally, it is important to mention that our construction is difficult to apply for the general case when the diffusion rates $d_i, i = 1, 2, 3, 4$, are not constant. We plan to continue working on these questions in our future work.

\bibliographystyle{plain}
\bibliography{ref}
\end{document}